\documentclass[11pt, reqno]{amsart}    

\usepackage{amsmath,amssymb,latexsym}
\usepackage{graphicx,xcolor,mathrsfs}
\usepackage{cite}

\newcommand{\R}{\mathbb{R}}
\newcommand{\T}{\mathbb{T}}

\newtheorem{theorem}{Theorem}[section]
\newtheorem{lemma}{Lemma}
\newtheorem{proposition}{Proposition}
\newtheorem{remark}{Remark}

\newtheorem*{main-theorem}{Main Theorem}
\newtheorem*{remark*}{Remark}
\newtheorem*{lemma*}{Lemma A.1}

\numberwithin{equation}{section}

\begin{document}

	\title[modified fKdV and fNLS equations]{Global dynamics of small solutions to the modified fractional Korteweg-de Vries and nonlinear Schr\"{o}dinger equations}

	\author{Jean-Claude Saut}
	
	\author{Yuexun Wang}
	
	\address{ Universit\' e Paris-Saclay, CNRS, Laboratoire de Math\'  ematiques d'Orsay, 91405 Orsay, France.}
	\email{jean-claude.saut@universite-paris-saclay.fr}
	
	\address{	
		School of Mathematics and Statistics, 
		Lanzhou University, 370000 Lanzhou, China.}
	\address{Universit\' e Paris-Saclay, CNRS, Laboratoire de Math\'  ematiques d'Orsay, 91405 Orsay, France.}
	
	\email{yuexun.wang@universite-paris-saclay.fr}
	
	\thanks{}

	
	\begin{abstract}
		This paper concerns the modified fractional Korteweg-de Vries (modified fKdV) and nonlinear Schr\"{o}dinger (modified fNLS) equations, with the dispersions \(|D|^{\alpha} \partial_x\) and \(|D|^{\alpha+1}\), respectively. We prove the global existence of small solutions for both the Cauchy problems to the modified fKdV and fNLS equations, with a modified scattering which has a logarithmic phase correction.  Our results cover the full range \(-1<\alpha<1,\ \alpha\neq 0\) for both the modified fKdV and fNLS equations. 
		
	\end{abstract}
	\maketitle

\section{Introduction}

We consider the modified fractional Korteweg-de Vries (modified fKdV) equation:
\begin{align}\label{eq:main-1}
	\partial_t u-|D|^{\alpha} \partial_x u\pm u^2\partial_xu=0,
\end{align}
where the unknown \(u\) maps \(\R_x\times\R_t\) to \(\R\), and \(|D|^{\alpha}\) is the usual Fourier multiplier operator with the symbol \(|\xi|^\alpha\),  and  the fractional nonlinear Schr\"odinger (fNLS) equation with cubic nonlinearity (modified fNLS): 
\begin{align}\label{eq:main-2}
\mathrm{i}\partial_t u-|D|^{\alpha+1} u\pm |u|^2u=0,
\end{align}
in which  the unknown \(u\) maps \(\R_x\times\R_t\) to \(\mathbb{C}\).

We are interested in \eqref{eq:main-1} and \eqref{eq:main-2} for \(\alpha\) in the range \(-1<\alpha<1,\ \alpha\neq 0\) and aim to study the global existence and modified scattering for small solutions of their Cauchy problems, respectively.

Those equations have rather limited physical interest but are very useful toy models to study the competition between the dispersive and nonlinear effects. For some specific values of $\alpha$ the dispersive term is reminiscent of the dispersion relation of physical problems {\it eg} $\alpha=-1/2$ in \eqref{eq:main-2} mimics the dispersion relation $\omega(k)=|k|^{1/2}$ of the water waves system in infinite depth, while $\alpha= -1/2$ (resp. $\alpha =1/2$) in \eqref {eq:main-1} mimics the dispersion of the so-called Whitham equation, (resp. Whitham with strong surface tension), see \cite{KLPS}, in the high frequencies limit. We have excluded for \eqref{eq:main-2} the case $\alpha=0$ that corresponds to the so-called Half-Wave equation which is not dispersive \footnote{We will nevertheless comment briefly below on the Half-Wave equation}. Fractional linear Schr\"{o}dinger operators  have been introduced in \cite{Laskin} but equations such as \eqref{eq:main-2} were not considered there.

Note however that \eqref{eq:main-2} is a particular case of the following equation
\begin{equation}\label{majda}
\mathrm{i}\psi_t=|D|^\sigma\psi\pm |D|^{-\beta/4}\left(\left||D|^{-\beta/4}\psi\right|^2|D|^{-\beta/4}\psi\right),
\end{equation}
where $\sigma>0$ and $\beta\in\R,$ introduced in \cite{MMT} (see also \cite{CMMT}) as a model for assessing the validity of weak turbulence theory for random waves. The parameter $\sigma$ controls the dispersion relation $\omega(k)=|k|^\sigma$ and $\beta$ the nonlinearity, in particular large values of $\beta$ makes the nonlinearity weaker because of a smoothing effect in \(x\). One recovers \eqref{eq:main-2} by choosing $\sigma=\alpha+1$ and $\beta=0$.

Equation \eqref{majda} is hamiltonian with Hamiltonian 
\begin{equation*}\label{HamMaj}
	H(\psi)=\frac{1}{2}\int_\R\left(\left||D|^{\sigma/2}\psi\right|^2\pm\frac{1}{2}\left||D|^{\beta/4}\psi\right|^4\right)\, dx.
\end{equation*}
As noticed in \cite{ZGPD} \eqref{majda} can be usefully written in Fourier space as
\begin{equation}\label{majzak}
	\mathrm{i}\frac{\hat \psi_k}{\partial t}=\omega(k)\hat \psi_k\pm \int T_{123k}\hat\psi_1\hat\psi_2\hat\psi_3^*\delta(k_1+k_2-k_3-k)\, dk_1dk_2dk_3,
\end{equation} 
where $\hat \psi_k=\hat\psi(k,t)$ is the $k-th$ Fourier coefficient of $\psi$  and $(^*)$ denotes complex conjugation.  In this form \eqref{majzak} looks like the so-called one-dimensional Zakharov equation with
dispersion relation
$$\omega(k)=|k|^\sigma,\; \sigma>0$$
and interaction coefficient
$$T_{123k}=T(k_1,k_2,k_3,k)=|k_1k_2k_3k|^{\beta/4}.$$ 
We refer to \cite{ZGPD} for a theoretical and numerical study of \eqref{majzak}.

\begin{remark}
	1. It is proven in \cite{KLS} that \eqref{eq:main-2} with $0<\alpha<1$ can be derived as a continuous limit of discrete nonlinear Schr\"{o}dinger equations with long-range lattice interactions.
	
	2. \eqref{eq:main-2} is reminiscent for $\alpha=-1/4$ of a one-dimensional {\it Full-dispersion} Davey-Stewartson system, see \cite{OS}.
\end{remark}

We now recall some useful properties of equations \eqref{eq:main-1}, \eqref{eq:main-2} together with known and conjectured results on the Cauchy problem. 

For both equations, there is a defocusing case (- sign in \eqref{eq:main-1} and \eqref{eq:main-2}), and a focusing case (+ sign in \eqref{eq:main-1} and \eqref{eq:main-2}) which leads to very different dynamics at least for large solutions.

In addition to the $L^2$ norm (mass), equations \eqref{eq:main-1} and \eqref{eq:main-2} conserve formally the Hamiltonian (energy), respectively
\begin{equation*}
H_\alpha(u)=\frac{1}{2}\int_\R\left(\big||D|^{\alpha/2}u\big|^{2}\mp \frac{1}{6}u^4\right)dx,
\end{equation*}
and 
\begin{equation*}
K_\alpha(u)=\frac{1}{2}\int_\R\left(\big||D|^{(\alpha+1)/2}u\big|^{2}\mp \frac{1}{2}|u|^4\right)dx.
\end{equation*}

By the Sobolev embedding $H^{1/4}(\R)\subset L^4(\R)$ this implies that $\alpha=1/2$ is the energy critical exponent for \eqref{eq:main-1} and $\alpha=-1/2$ is the energy critical exponent for \eqref{eq:main-2}. 

On the other hand,  \eqref{eq:main-1} is invariant under the scaling transformation 
$$u_\lambda(x,t)=\lambda^{\alpha/2} u(\lambda x, \lambda^{\alpha +1}t),$$ 
which implies that $\alpha=1$ is the $L^2$ critical exponent.  

Similarly, \eqref{eq:main-2} is invariant under the transformation
$$u_\lambda(x,t)= \lambda^{(\alpha+1)/2}u(\lambda x, \lambda^{\alpha+1} t)$$ implying that $\alpha =0$ is the $L^2$ critical exponent in this case.

In the defocusing  case and when $\alpha>0$    one has  a formal conservation of the energy space  $H^{\alpha/2}(\R)$ for \eqref{eq:main-1} and when $\alpha>-1$ of the energy space $H^{(\alpha+1)/2}(\R)$ in the case of equation \eqref{eq:main-2}.

For both equations in the defocusing case, a standard compactness method implies the existence of global weak solutions of the Cauchy problem  in the energy space $H^{\alpha/2}(\R)$ when $\alpha>1/3$ (resp. $H^{(\alpha+1)/2}(\R)$ when $\alpha>-2/3$), this condition ensuring that the embedding $H^{\alpha/2}(\R)\subset L^3_{\text{loc}}(\R)$ (resp. $H^{(\alpha+1)/2}(\R)
\subset L^3_{\text{loc}}(\R)$) is compact.

Concerning the local Cauchy problem, it is (without any use of dispersive estimates) trivially locally well-posed in $H^s(\R)$ where $s>3/2$ for  \eqref{eq:main-1}  and $s>1/2$ for \eqref{eq:main-2}. Of course much better results are expected when using the dispersive properties of the equations. This has been done for instance in \cite{LPS, MPV} for the quadratic fKdV equation and it is likely that the methods there could be used to get similar results for \eqref{eq:main-1}. Since we focus in the present paper on the global existence of small solutions we will not consider this issue here but we will comment below on the improved local Cauchy theory for \eqref{eq:main-2}.

Note that in the particular case $\alpha =1$ (modified Benjamin-Ono equation) the Cauchy problem was proven in \cite{KT} to be locally well-posed in $H^s(\R), s\geq 1/2$ leading to the global well-posedness in the defocusing case. The large time behavior of solutions, probably scattering, is an open problem.

Most of global well-posedness results with arbitrary initial data issues for  \eqref{eq:main-1} are still open in the range $0<\alpha \leq 1$ with the exception of the case $\alpha =1$ where as previously noticed, Kenig and Takaoka  \cite{KT} proved the global well-posedness in the defocusing case while Martel and Pilod \cite{MP} established the finite time blow-up in the focusing case. 

\begin{remark}
When $1<\alpha\leq 2,$ which is outside the range we focus on, Guo \cite{Guo} proved that the Cauchy problem for \eqref{eq:main-1} is locally well-posed in $H^s(\R), s>\frac{3-\alpha}{4}$ and globally well-posed in $H^s(\R), s>\frac{\alpha}{2}.$
\end{remark}

An important role in the long time dynamics of equations of the focusing equations \eqref{eq:main-1} and \eqref{eq:main-2} is played by the ground state solutions. In the case of \eqref{eq:main-1} there are solutions of the form $Q_c(x-ct), c>0$ so that they satisfy the equation:
\begin{equation*}
cQ_c+|D|^\alpha Q_c-\frac{Q_c}{3}=0.
\end{equation*}

While no such solution exist in the defocusing case, the existence in the focusing case when $\alpha>1/2$ is standard (see for instance \cite{LPS3}) and the uniqueness of a positive and  even such solution  is established in the same range in \cite{FL}. Moreover, the method used in \cite{LPS3} to establish the orbital stability of $L^2$ subcritical ground states  for  the quadratic fKdV equation  allows to prove the orbital stability of ground state solutions of \eqref{eq:main-1} in the $L^2$ subcritical case $\alpha>1$ that is outside of the range of $\alpha's$ that we consider here.

Next, we note that
$$Q_c(x)=\sqrt cQ(c^{1/\alpha}x),$$
where $Q=Q_1.$ This implies that 
$$\|Q_c\|_{L^2}^2=c^{\frac{\alpha-1}{\alpha}}\|Q\|_{L^2}^2,$$
proving that \eqref{eq:main-1} has ground states of arbitrary $L^2$ norm by taking arbitrary large velocities when $1/2<\alpha<1$ (resp. arbitrary small velocities when $\alpha >1$). This excludes scattering of small solutions in the $L^2$ norm.

Similarly,
$$\|D^{\alpha/2}Q_c\|^2_{L^2}=c^{\frac{3\alpha-2}{2\alpha}}\|D^{\alpha/2}Q\|^2_{L^2},$$
proving for instance that no scattering in the energy space is possible when $1/2<\alpha<1$.

For the values $0<\alpha<1$ the numerical simulations in the recent paper \cite{KSW} lead to the following conjectures:
\begin{itemize}
	\item Defocusing case for \eqref{eq:main-1}:
	
	(i) When $\alpha\geq 1/2$, initial data of arbitrary mass lead to solutions which will be dispersed and are global in time, leading to a dynamics similar to that of the defocusing generalized KdV equation, see\cite{FLPV}.
	
	(ii) When $0<\alpha<1/2$ initial data of sufficiently large mass lead to the formation of a singularity (cusp) in finite time.
	
	(iii) When $0<\alpha<1$ the solutions of sufficiently small mass remain smooth for all \(t\).
	\item Focusing case for \eqref{eq:main-1}:
	
	(i) When $\alpha>1$ the solutions of finite energy are global and decompose into ground states as $t\to+\infty$.
	
	(ii) When $0<\alpha<1$ the solution of sufficiently large mass blows up in finite time, the nature of blow-up being different in the $L^2$ energy subcritical case $1/2<\alpha<1$ and energy supercritical case $0<\alpha<1/2$.
	
	(iii) When $0<\alpha<1$ the solutions of sufficiently small mass remain smooth for all \(t\).

\end{itemize}

We now turn to the Cauchy problem for \eqref{eq:main-2}. The known results are natural extensions on the Cauchy theory of the nonlinear Schr\"{o}dinger equation, see \cite{Ca} but with extra difficulties due the the weak dispersion, in particular the Strichartz estimates involve a loss of derivatives.

The Cauchy problem for general fractional nonlinear Schr\"{o}dinger equations including \eqref{eq:main-2} as a particular case has been studied in many works, see for instance  \cite{HS, CHKL, D1, D2, D3} and the references therein and the numerical simulations  in \cite{KSM}. \footnote{We recall that the Cauchy problem is trivially locally well-posed in $H^s(\R), s>1/2$ for any $\alpha>-1$.}

We recall that  for any $\lambda>0$ \eqref{eq:main-2} is invariant by the scaling 
$$u(x,t)\mapsto u_\lambda(x,t)=\lambda^{(\alpha+1)/2}u(\lambda x,\lambda^{\alpha+1}t),$$
so that
$$\|u_\lambda\|_{\dot{H}^\gamma}=\lambda^{\gamma+\frac{\alpha}{2}}\|u\|_{\dot{H}^\gamma}$$
and $s_c=-\alpha/2$ is the scaling-critical regularity exponent. One can therefore expect local well-posedness in $H^s(\R), s\geq s_c.$

Another important exponent is 
$$s_g=\frac{1-\alpha}{4},$$
which is the critical exponent for the pseudo-Galilean invariance, see \cite{HS}.

One has the following results and conjectures (motivated in particular by the numerical simulations in \cite{KSM})  for the Cauchy problem:
\begin{itemize}
	
	\item Local well-posednes (see \cite{HS}). The Cauchy problem for \eqref{eq:main-2} is locally well-posed in $H^s(\R), s\geq s_g.$  Assume that $\alpha\in (-1/2,1)$. Then the Cauchy problem for \eqref{eq:main-2} is ill-posed in $H^s(\R), s\in (s_c,0).$  The existence proof uses Strichartz estimates with loss. A similar result was obtained in \cite{CHKL} but only in the case  $0<\alpha<1.$ The proof uses Bourgain's type spaces and allows to treat the periodic case.
	
	\item Ill-posedness, see \cite{CP}. The local Cauchy problem is ill-posed in $H^s(\R)$ in the sense that the flow-map $u_0\mapsto u(\cdot,t)$ is not continuous at $0$ in $H^s(\R)$ in the following cases :
	
	(i) $-1<\alpha<0,\; s<0$.
	
	(ii) $0<\alpha<1,\; s<-\alpha/2$.
	
	(iii) $\alpha=1,\;s<-1/2$.

	Ill-posedness results for the Half-Wave equation can also be found in \cite{CP} and in\cite{GTV} for the periodic Half-Wave equation.
	
	\item Global well-posedness is expected in the energy subcritical defocusing case $\alpha>-1/2$ while finite time blow-up is expected in the energy critical and supercritical case $-1<\alpha<-1/2$. Note that when $\alpha>0$ the trivial  local well-posedness in  $H^{(\alpha+1)/2}(\R)$ together with the conservation of energy imply the global well-posedness in  $H^{(\alpha+1)/2}(\R)$ in the defocusing case.

	\item Finite time blow-up is expected in the focusing case in the mass critical or super critical case $-1<\alpha<0$ as supported by the numerical simulations in \cite{KSM}. This appears to be an open problem, but it was solved for the Half-Wave equation ($\alpha=0$) which corresponds to the mass critical case, see \cite{KLR} and the survey \cite{Him}.
	
	\item More precisely, Krieger, Lenzmann and Rapha\"{e}l proved for the focusing Half-Wave equation the existence of a family of traveling waves with subcritical arbitrary small mass. They also proved the existence of a minimal mass $H^{1/2}$ solution that blows up in finite time.

We also mention \cite{GLPR} where one constructs for the Half-Wave equationan asymptotic global-in-time compact two-soliton solution with
arbitrarily small $L^2$-norm which exhibits the following two regimes: (i) a transient
turbulent regime characterized by a dramatic and explicit growth of its
$H^1$-norm on a finite time interval, followed by (ii) a saturation regime in which
the $H^1$-norm remains stationary large forever in time.

 Concerning the defocusing Half-Wave equation, Pocovnicu \cite{Poco} gives an example of initial data leading to a global solution with growing higher Sobolev norms and proves the approximation of solutions corresponding to small initial data supported on  positive frequencies by the Szeg\"{o} equation.  A similar result has been established by G\' erard and Grellier \cite{GG} in the periodic case.
\end{itemize}

 We briefly describe two related results for fNLS equations. In \cite{Lan}, Lan proved a finite blow-up result for the fractional nonlinear Schr\"{o}dinger equation:
 \begin{equation*}
 \mathrm{i}u_t-|D|^\beta u+|u|^{2\beta}u=0,\quad 1\leq \beta<2.
 \end{equation*}
In \cite{BHL} the finite time blow-up for \eqref{eq:main-2} posed on a finite interval with Dirichlet boundary conditions is established when $0<\alpha<1$ in the focusing case.

Ground state solutions are expected to play a major role in the dynamics of the focusing case. There are solutions of the form
$$u(x,t)=e^{\mathrm{i}\omega t}Q_\omega (x),$$
where \(\, \omega>0,\;Q_\omega \in H^{(\alpha+1)/2}(\R)\), 
so that (taking by simplicity + sign)
\begin{equation*}
\omega Q_\omega+ |D|^{\alpha+1} Q_\omega-|Q_\omega|^2Q_\omega=0.
\end{equation*}

As previously 
$$Q_\omega(x)=\sqrt \omega Q(\omega^{1/(\alpha +1)}(x)),$$
where $Q=Q_1$, so that 
$$\|Q_\omega\|^2_{L^2}=\omega^{\frac{\alpha}{\alpha +1}}\|Q\|^2_{L^2},$$
proving that bound states with arbitrary small $L^2$ norm are possible.

\begin{remark}
	As noticed in \cite{HS2}, and contrary to the case of the classical cubic NLS equation, one cannot construct a traveling wave solution simply by boosting  a static one, since the equation does not have any exact Galilean invariance due to the non-locality of the fractional Laplacian. Motivated by a {\it pseudo-Galilean invariance} of \eqref{eq:main-2}, Hong and Sire (\cite{HS2}) proposed to look for a solution of the form:
	$$u_{\omega,k}(t,x)=e^{-\mathrm{i}t(|k|^{\alpha+1}-\omega^{\alpha+1})}e^{\mathrm{i}kx}Q_{\omega,k}(x-2t(\alpha+1)|k|^{\alpha-1}k),$$
	which will lead to  a natural family of moving solitary waves with frequency $\omega$ and speed $k$. The profile $Q_{\omega,k}$ then solves the equation:
	\begin{equation}\label{HS}
	\mathcal P_kQ_{\omega,k}+\omega^{\alpha-1}Q_{\omega,k}-|Q_{\omega,k}|^2Q_{\omega,k}=0,
	\end{equation}
	where 
	$$\mathcal P_k=e^{-\mathrm{i}kx}|D|^{(\alpha+1)/2}e^{\mathrm{i}kx}-|k|^{\alpha+1}+\mathrm{i}(\alpha+1)|k|^{\alpha-1}k\cdot\nabla_x.$$
	
	It is proven in \cite[Theorem 1.1]{HS2},  that for any $k\in \R$ , there exists $Q_{\omega,k}\in H^1(\R)$ solving \eqref{HS} for some $\omega>0$. Moreover $Q_{\omega,k}\in C^\infty(\R)$.
\end{remark}

\begin{remark}
We mention here some facts on the periodic problem for \eqref{eq:main-2}. Thirouin \cite{Thi} proved that the periodic Cauchy problem for the defocusing \eqref{eq:main-2} is globally well posed in $C^\infty(\T)$ when $-1/3<\alpha<1, \alpha\neq 0$. Moreover the Sobolev norms grow at most at a polynomial rate.

Recent results concern the stochastic periodic problem. Namely, for the defocusing fNLS equation \eqref{eq:main-2} posed on the circle $\T$  and with initial data distributed according to the Gibbs measure,  Sun and Tzvetkov \cite{SunTz, SunTz2} constructed strong solutions for $\alpha>\alpha_0=\frac{17-\sqrt{233}}{14}\sim 0.124.$

It is moreover established in \cite{SunTz2} that the Cauchy problem cannot be solved by a Picard iteration on the framework of Bourgain's spaces when $\alpha<1/5$.
\end{remark}

\vspace{0.3cm}
Let us  now describe  the  goal of the present paper, that is the global existence and scattering of small solutions of the Cauchy problem for \eqref{eq:main-1} and   \eqref{eq:main-2}. This completes previous works of the Authors on the modified fKdV equation, \cite{SW} where shock formation was proven when $-1\leq\alpha<0$, \cite{KSW} which focused on issues arising from large initial data and \cite{SW1} where long time existence and   scattering of small solutions was established in the case $-1<\alpha<0.$
In the case of equation \eqref{eq:main-2} our results extend previous known results in the sense that they cover the full range $-1<\alpha<1, \alpha\neq 0$.

To study the Cauchy problem, one has to impose the initial data
\begin{align}\label{eq:initial-1}
u(x,0)=u_0(x).
\end{align}

Now we are in a position to state our first main result which is the global existence and scattering of small solutions on the modified fKdV equation in the focusing case \footnote{A similar result holds in the defocusing case with minor sign changes.}:
\begin{theorem}\label{th:main-1} Let \(\alpha\in(-1,1)\setminus\{0\}\). Define the profile of \(u\)
	\[f(t)=e^{-t|D|^{\alpha} \partial_x}u(t),\]
	and the \(Z\)-norm
	\[\|f\|_Z=\|(|\xi|^{(1-\alpha)/4}+|\xi|^{10})\widehat{f}(\xi)\|_{L^\infty_\xi}.\]
	Assume that \(N_0=100,\ p_0\in (0,-10^{-3}\alpha] \ \text{when} -1<\alpha<0\), and \(\ p_0\in (0,10^{-3}(1-\alpha)] \ \text{when}\ 0<\alpha<1\) are fixed, and \(u_0\in H^{N_0}(\mathbb{R})\) satisfies
	\begin{align}\label{fKdV1}
	\|u_0\|_{H^{N_0}}+\|u_0\|_{H^{1,1}}+\|u_0\|_Z=\varepsilon_0\leq \bar{\varepsilon},
	\end{align}
	for some constant \(\bar{\varepsilon}\) sufficiently small (depending only on \(\alpha\) and \(p_0\)). Then the Cauchy problem \eqref{eq:main-1} and \eqref{eq:initial-1} admits a unique
	global solution \(u\in C(\mathbb{R}:H^{N_0}(\mathbb{R}))\) satisfying the following uniform bounds for \(t\geq 1\)
	\begin{align}\label{fKdV2}
	t^{-p_0}\|u\|_{H^{N_0}}+t^{-p_0}\|f\|_{H^{1,1}}+\|f\|_Z\lesssim \varepsilon_0,
	\end{align}
	and
	\begin{align}\label{fKdV-decay}
	\|u\|_{L^\infty}+\|\partial_xu\|_{L^\infty}\lesssim \varepsilon_0 t^{-1/2}.
	\end{align}
	Moreover,  there exists \(w_\infty\in L^\infty(\mathbb{R})\) such that  for \(t\geq 1\)
	\begin{align}\label{fKdV3}
	t^{p_0}\left\|\exp\left(\frac{\mathrm{i}\xi|\xi|^{1-\alpha}}{|\alpha|(\alpha+1)}\int_1^t|\widehat{f}(\xi,s)|^2\,\frac{d s}{s}\right)(|\xi|^{(1-\alpha)/4}+|\xi|^{10})\widehat{f}(\xi)-w_\infty(\xi)\right\|_{L^\infty_\xi}\lesssim \varepsilon_0.
	\end{align}
\end{theorem}

Our second main result is the global existence and scattering of small solutions on the modified fNLS equation in the focusing case, which is stated precisely as follows \footnote{A similar result holds in the defocusing case with minor sign changes.}:
\begin{theorem}\label{th:main-2} Let \(\alpha\in(-1,1)\setminus\{0\}\). Define the profile of \(u\)
	\[f(t)=e^{\mathrm{i}t|D|^{\alpha+1}}u(t),\]
	and the \(Z\)-norm 
	\[\|f\|_Z=\|(|\xi|^{(1-\alpha)/4}+|\xi|^{10})\widehat{f}(\xi)\|_{L^\infty_\xi}.\] 
	Assume that \(N_0=100,\ p_0\in (0,-10^{-3}\alpha]\ \text{when} -1<\alpha<0\), and \(\ p_0\in (0,10^{-3}(1-\alpha)] \ \text{when}\ 0<\alpha<1\) are fixed, and \(u_0\in H^{N_0}(\mathbb{R})\) satisfies
	\begin{align}\label{fNLS1}
	\|u_0\|_{H^{N_0}}+\|u_0\|_{H^{1,1}}+\|u_0\|_Z=\varepsilon_0\leq \bar{\varepsilon},
	\end{align}
	for some constant \(\bar{\varepsilon}\) sufficiently small (depending only on \(\alpha\) and \(p_0\)). Then the Cauchy problem \eqref{eq:main-2} and \eqref{eq:initial-1} admits a unique
	global solution \(u\in C(\mathbb{R}:H^{N_0}(\mathbb{R}))\) satisfying the following uniform bounds for \(t\geq 1\)
	\begin{align}\label{fNLS2}
	t^{-p_0}\|u\|_{H^{N_0}}+t^{-p_0}\|f\|_{H^{1,1}}+\|f\|_Z\lesssim \varepsilon_0,
	\end{align}
	and
	\begin{align}\label{fNLS-decay}
	\|u\|_{L^\infty}+\|\partial_xu\|_{L^\infty}\lesssim \varepsilon_0 t^{-1/2}.
	\end{align}
	Moreover,  there exists \(w_\infty\in L^\infty(\mathbb{R})\) such that  for \(t\geq 1\)
	\begin{align}\label{fNLS3}
	t^{p_0}\left\|\exp\left(\frac{-\mathrm{i}|\xi|^{1-\alpha}}{|\alpha|(\alpha+1)}\int_1^t|\widehat{f}(\xi,s)|^2\,\frac{d s}{s}\right)(|\xi|^{(1-\alpha)/4}+|\xi|^{10})\widehat{f}(\xi)-w_\infty(\xi)\right\|_{L^\infty_\xi}\lesssim \varepsilon_0.
	\end{align}
\end{theorem}

The  kind of  problems under study here are  usually referred to as  ``long-range" scattering issues. This phenomena was first pointed out by Ozawa \cite{Ozawa} in the context of the cubic one-dimensional Schr\"{o}dinger equation ($\alpha=1$ in \eqref{eq:main-2}) means that  the solutions of the nonlinear equation do not scatter to the free solutions. For \eqref{eq:main-1} (\eqref{eq:main-2}), the long-range phenomenon can be seen from the fact that \(\|u^2u_x(t)\|_{L^2}\) (\(\||u|^2u(t)\|_{L^2}\)) fails to be integrable for all time. In fact \(\int_1^t\|u^2u_x(s)\|_{L^2}\,\ d s<\infty\) (\(\int_1^t\||u|^2u(s)\|_{L^2}\,\ d s<\infty\)) only holds true till \(t\thickapprox e^{c\varepsilon^{-2}}\) for small solutions of size  \(\varepsilon>0\), which means that one can only expect that the solutions of \eqref{eq:main-1} (\eqref{eq:main-2}) behave like the free solutions at most till to \(t\thickapprox e^{c\varepsilon^{-2}}\). In other words, \eqref{eq:main-1} (\eqref{eq:main-2}) still possesses a global solution, however this solution differs from the free solution, more precisely, by a logarithmic correction in scattering (modified scattering). 

Many results of this type were obtained for the modified KdV, modified Benjamin-Ono and cubic nonlinear Schr\"{o}dinger equations, see for instance \cite{MR3519470, MR3462131, HN0, IT1, MR3121725} and the references therein. On the other hand it was proven in \cite{BGLV} that small data scattering fails to hold for the focusing Half-Wave equation due to the existence of solitary waves of small speed.

In the context of the present paper, the modified scattering for equation \eqref{eq:main-2} with $\alpha=-1/2$ was first obtained in \cite{MR3121725} and then extended to other values of $\alpha$, namely  $-1<\alpha\leq 1/2, \alpha\neq 0$ in a series of papers \cite{HN2, HN3, Nau, HN5}.

The present work is based on the space-time resonance argument and its variants developed in different contexts, see for instance \cite{GMS,GNT,HN1,MR3121725} and the references therein, and close to \cite{MR3121725} in methodology, but mainly inspired by \cite{SW1}.
Our method can handle the modified fKdV and fNLS equations together, and can treat the full range of values of $-1<\alpha<1, \alpha\neq 0$ uniformly. Even for the modified fNLS equation, our proof is much simpler than the existing ones \cite{MR3121725, HN2, HN3, Nau, HN5}. 
Let us now explain our main improvements. 
In the previous work \cite{SW1}, in order to handle \eqref{eq:main-1}, our crucial observations were that one can use the structure of the nonlinearity (the presence of a derivative) to eliminate part of resonances in low frequencies, and  so work with a more natural norm \(\|\partial_t\widehat{f_k}(\cdot,t)\|_{L^2}\) to replace \(\|\partial_t\widehat{f_k}(\cdot,t)\|_{L^\infty}\) (the latter one was used in \cite{MR3121725} to handle \eqref{eq:main-2} with $\alpha=-1/2$), which allows to extend the estimates to the whole interval \(-1<\alpha<0\). However, these observations are not enough to treat the case \(0<\alpha<1\) for \eqref{eq:main-1}. To prove Theorem \ref{eq:main-1} for \eqref{eq:main-1} when \(0<\alpha<1\), we need to work with a new Z-norm and use a more delicate localization on the \(L^2\) norm (see \eqref{l18}) and find a more subtle division on the frequencies to handle the resonance by using integration by parts either on space or on time. These new observations allow us to prove the global existence and modified scattering for small solutions of \eqref{th:main-1} not only for \(-1<\alpha<0\) but also for \(0<\alpha<1\).  Via these observations, we find that there is a correspondence between \eqref{eq:main-2} and \eqref{eq:main-1} in the  resonance analysis, in fact the resonance analysis in \eqref{eq:main-2} corresponds to some member of the eight cases in \eqref{eq:main-1}. This leads us to give a parallel proof of Theorem \ref{eq:main-2} for \eqref{eq:main-2} in  both cases \(-1<\alpha<0\) and \(0<\alpha<1\).

We next explain the strategy of the proofs of  Theorem \ref{th:main-1} and \ref{th:main-2}. 

The local well-posedness on the time interval \([0,1]\) for the equation \eqref{eq:main-1} (\eqref{eq:main-2}) with initial data \eqref{eq:initial-1}  is standard provided $\|u_0\|_{H^2}$ is sufficiently small, in particular  under the smallness assumption \eqref{fKdV1} (\eqref{fNLS1}). 

Then the existence and uniqueness of global solutions will be constructed by a bootstrap argument which allows   to extend the local solutions to  global ones. More precisely, we assume that the \(X\)-norm 
\begin{equation}\label{l8}
\begin{aligned}
\|u\|_{X}&=\sup_{t\geq 1}\bigg(t^{-p_0}\|u\|_{H^{N_0}}+t^{-p_0}\|f\|_{H^{1,1}}+\|f\|_Z\bigg)\leq \varepsilon_1,
\end{aligned}
\end{equation}
is a priori small and then aim to show that the a priori assumption \eqref{l8} may be improved to
\begin{align}\label{l9}
\|u\|_{X}\leq C(\varepsilon_0+\varepsilon_1^3),
\end{align}
for some absolute constant \(C>1\) and \(\varepsilon_1=\varepsilon_0^{1/3}\).

The following remarks will be helpful to understand Theorem \ref{th:main-1} and \ref{th:main-2}.

\begin{remark}
	The choice of spaces in Theorems \ref{th:main-1} and \ref{th:main-2} eliminates the possibility of existence of solitary waves with arbitrary small energy norms in the focusing case.
\end{remark}

\begin{remark}
	Equation \eqref{eq:main-2} has a multi-dimensional version, however,  its global existence and scattering issue for small solutions is much easier than the one-dimensional case considered in this paper since the solutions have a much faster decay in time in the multi-dimensional case. 
	The scattering of small solutions to the defocusing $L^2$ supercritical fractional NLS equation has been studied in \cite{GZ} in higher dimensions. 
\end{remark}

\begin{remark}
	We point out that in fact one does not need the decay rate \(t^{-1/2}\) of \(\|\partial_xu\|_{L^\infty}\) in the proof of Theorem \ref{th:main-2}. This decay rate is a by product of \eqref{fNLS2}. 
\end{remark}

\begin{remark}
	The results for \eqref{eq:main-1} and \eqref{eq:main-2} can be  extended to a wider class of equations with short range perturbations of the nonlinearity, namely 
	\begin{align*}
	\mathcal{N}(u)=\mp u^2\partial_xu+\mu u^p\partial_xu
	\end{align*}
for \eqref{eq:main-1}, and
    \begin{align*}
    \mathcal{N}(u)=\mp |u|^2u+\mu|u|^pu
    \end{align*}
for \eqref{eq:main-2}, where \(\mu\in\R\) and \(p>2\).
\end{remark}

\begin{remark}
	Actually our result for \eqref{eq:main-2} can be extended 
	to a class of nonlinearities of the form
	\begin{align}\label{new nonlinearity}
	\mathcal{N}(u)=\mp |u|^2u+c_1u^3+c_2u\bar{u}^2+c_3\bar{u}^3
	\end{align}
	where $c_1, c_2, c_3 \in \mathbb{C}$. In fact, there are four cases in \eqref{eq:main-2} with \eqref{new nonlinearity}:
	\[\{(+,+,-),(+,+,+),(+,-,-),(-,-,-)\}\] 
	which correspond to some four members of the eight cases
	\begin{align*}
	\{(+,+,+),(+,-,-),(-,+,-),(-,-,+),\\
	(+,+,-),
	(+,-,+),(-,+,+),(-,-,-)\}
	\end{align*} 
	in \eqref{eq:main-1} in the resonance analysis, hence the argument for \eqref{eq:main-1} can be used for \eqref{eq:main-2} with \eqref{new nonlinearity}.
	But since those nonlinearities have no great physical relevance we refrain to go into detail for this extension and focus on the classical cubic nonlinearity.
\end{remark}

\begin{remark}
	The result for \eqref{eq:main-2} can also be extended to a modified derivative fNLS equation, i.e.,  
	the nonlinearity takes the form
	\begin{align*}
	\mathcal{N}(u)=\mp \mathrm{i}\partial_x(|u|^2u).
	\end{align*}
	This is because the derivative in the nonlinearity does not play a role in our argument which allows to deal with the modified fKdV and fNLS equations together in this paper. (Recall that the structure of the nonlinearity (the presence of a derivative) plays a crucial role to handle \eqref{eq:main-1} when \(-1<\alpha<0\) in \cite{SW1}.) 
\end{remark}

\begin{remark}
	The  argument in the present  paper can handle \eqref{eq:main-1} and \eqref{eq:main-2} for both of \(-1<\alpha<0\) and \(0<\alpha<1\) simultaneously. 
	This improvement is very important since it is expected to  apply to a much wider class of  dispersive equations whose dispersive relation exhibits different scales at low  and high frequencies, for instance the modified Whitham equation or some {\it full dispersion} equations, see  \cite{La1, OS}.  
	We plan to come back to this issue in a forthcoming work. 
\end{remark}

\begin{remark}
	As it is the case in all previous works concerning the scattering of small solutions, the methods used to prove the above Theorems do not distinguish between the focusing and defocusing case, so that the results are far from being optimal in the defocusing case when $\alpha>1/2$ (resp. $\alpha>-1/2$) since global existence and dispersion for arbitrary large solutions is likely to occur in the energy subcritical regime.
	Note however that proving such a result for the defocusing \eqref{eq:main-1}, \eqref{eq:main-2} or more generally defocusing nonlinear dispersive equations is a delicate  question. We refer for instance to \cite{Dod, FLPV} for the $L^2$ critical and supercritical defocusing generalized KdV equation.   A striking result using complete integrability  concerns the defocusing Davey-Stewartson II system \cite {Su4, Pe, NRT}, the later paper proving global well-posedness and scattering for arbitrary initial conditions in $L^2.$ On the other hand, in the focusing case, those methods exclude the possibility of small solitary waves in the $L^2$ subcritical cases by choosing an appropriate norm, eliminating thus a  nonlinear dynamics involving the emergence of arbitrary small solitary waves that excludes scattering. 
	
	Nevertheless those scattering results are in some sense optimal when finite time blow-up for large initial data is expected, for instance in the focusing $L^2$ critical or supercritical cases and in the defocusing energy supercritical cases. On the other hand, this kind of global existence and scattering result for small solutions in the present paper provides more complete information on the long time asymptotic behavior than the usual global well-posedness result for large solutions..   
\end{remark}

Next we list some notations frequently used throughout the paper.

\noindent{\bf{Notation and conventions.}} 
Let \(L^p(\R)\) (\(p \in [1,\infty]\)) be the standard Lebesgue spaces, in particular, \(L^2(\R)\) is a Hilbert space with inner product 
\[
(g,h)_2:=\int_{\R}gh \, d x.
\]
Similarly, let $H^s(\R)$ (\(s >0\)) be the usual Sobolev spaces with norm 
\[
\|g\|_{H^s(\R)}: = \|(1-\partial_x^2)^{s/2} g \|_{L^2(\R)},
\]
and  let $C([0,T]: H^s(\R))$ be the space of all bounded continuous functions $g\colon [0,T]\rightarrow H^s(\R)$ normed by
\[
\| g \|_{C([0,T]: H^s(\R))}: =  \sup_{t \in [0,T]} \|g(t,\cdot)\|_{H^s(\R)}. 
\]

We denote by \(\mathcal{F}(g)\) or \(\widehat{g}\) the Fourier transform of a Schwartz function \(g\) whose formula is given by 
\begin{equation*}
\begin{aligned}
\mathcal{F}(g)(\xi)=\widehat{g}(\xi):=\frac{1}{\sqrt{2\pi}}\int_{\mathbb{R}}g(x)e^{-\mathrm{i}x\xi}\,d x
\end{aligned}
\end{equation*}
with inverse
\begin{equation*}
\begin{aligned}
\mathcal{F}^{-1}(g)(x)=\frac{1}{\sqrt{2\pi}}\int_{\mathbb{R}}g(\xi)e^{\mathrm{i}x\xi}\, d \xi,
\end{aligned}
\end{equation*}
and by \(m(\partial_x)\) the Fourier multiplier with symbol \(m\) via the relation 
\begin{equation*}
\begin{aligned}
\mathcal{F}\big(m(\partial_x)g\big)(\xi)=m(\mathrm{i}\xi)\widehat{g}(\xi).
\end{aligned}
\end{equation*}

Take \(\varphi\in C_0^\infty(\mathbb{R})\) satisfying \(\varphi(\xi)=1\) for \(|\xi|\leq 1\) and \(\varphi(\xi)=0\) when \(|\xi|>2\), and
let 
\begin{equation*}
\begin{aligned}
\psi(\xi)=\varphi(\xi)-\varphi(2\xi),\quad \psi_j(\xi)=\psi(2^{-j}\xi),\quad \varphi_j(\xi)=\varphi(2^{-j}\xi),
\end{aligned}
\end{equation*}
we then may define the  
Littlewood-Paley projections \(P_j,P_{\leq j},P_{> j}\)  via 
\begin{equation*}
\begin{aligned}
\widehat{P_jg}(\xi)=\psi_j(\xi)\widehat{g}(\xi),\quad \widehat{P_{\leq j}g}(\xi)=\varphi_j(\xi)\widehat{g}(\xi),\quad P_{> j}=1-P_{\leq j},
\end{aligned}
\end{equation*}
and also \(P_{\sim j},P_{\lesssim j},P_{\ll j}\) by 
\begin{equation*}
\begin{aligned}
P_{\sim j}=\sum_{2^k\sim 2^j}P_k, \quad P_{\lesssim j}=\sum_{2^k\leq 2^{j+C}}P_k,\quad P_{\ll j}=\sum_{2^k\ll 2^j}P_k,
\end{aligned}
\end{equation*}
and the obvious notation for \(P_{[a,b]}\).
We will also denote \(g_j=P_jg, g_{\lesssim j}=P_{\lesssim j}g\), and so on, for convenience.

The notation \(C\)  always denotes a nonnegative universal constant which may be different from line to line but is
independent of the parameters involved. Otherwise, we will specify it by  the notation \(C(a,b,\dots)\).
We write \(g\lesssim h\) (\(g\gtrsim h\)) when \(g\leq  Ch\) (\(g\geq  Ch\)), and \(g \sim h\) when \(g \lesssim h \lesssim g\).
We also write \(\sqrt{1+x^2}=\langle x\rangle\), \(\|g\|_{H^{1,1}}=\|\langle x\rangle g\|_{H^1}\), and \(P_{[k-2,k+2]}:=P_k^\prime\) for simplicity.

The paper is organized as follows. Sect.\ref{Proof of mfKdV} is devoted to the proof of Theorem \ref{th:main-1}. We first deduce the decay estimates of the solutions of \eqref{eq:main-1} and then use them to prove the main estimates, namely to control the bounds in \eqref{l9}. Showing the Z-norm bound constitutes the main body of the remaining proof. In Sect.\ref{Proof of mfNLS}, we will show Theorem \ref{th:main-2}, by adopting a similar route as in Sect.\ref{Proof of mfKdV}.

\section{Proof of Theorem \ref{th:main-1}}\label{Proof of mfKdV}

\subsection{Decay estimates}

We have the following decay estimates for the solutions of \eqref{eq:main-1}:
\begin{lemma}\label{decay:fKdV}
	Let \(\alpha\in (-1,1)\setminus\{0\}\) and \(t\geq 1\). Assume that \(u\) is the solution of the equation \eqref{eq:main-1} and satisfies
	\begin{align*}
	t^{-p_0}\|u\|_{H^{N_0}}+(1+t)^{-p_0}\|f\|_{H^{1,1}}+\|f\|_Z\leq 1,
	\end{align*}
	then it holds
	\begin{align}\label{l6}
	\|u\|_{L^\infty}+\|\partial_xu\|_{L^\infty}\lesssim t^{-1/2}.
	\end{align}
	
\end{lemma}
\begin{proof} The proof for \(\alpha\in (-1,0)\) is included in \cite[lemma 2.2]{SW1}, so we only focus on the case of \(\alpha\in (0,1)\) in the following. Observe that 
	\(\frac{1-\alpha}{2}\in(0,\frac{1}{2})\) and \(-\frac{1+3\alpha}{4}\in(-1,-\frac{1}{4})\). 
	In the frequency regime \(2^k\geq t^{(1-4p_0)/5}\),  we use \eqref{l4} and
	\eqref{ap1} to deduce that 
	\begin{equation}\label{high frequency}
	\begin{aligned}
	\quad 2^{3k/2}\left\|P_ku\right\|_{L^\infty}
	&\lesssim t^{-\frac{1}{2}}2^{3 k/2}2^{(1-\alpha)k/2}2^{-k/2}2^{-N_0k/2}\|P_k^\prime f\|_{H^{N_0}}^{1/2}\\
	&\quad\times(\|\widehat{P_k^\prime f}\|_{L^2}+2^k\|\partial\widehat{P_k^\prime f}\|_{L^2})^{1/2}\\
	&\lesssim t^{-\frac{1}{2}}t^{(1-4p_0)(3-\alpha-N_0)/10}t^{p_0}\lesssim t^{-\frac{1}{2}}.
	\end{aligned}
	\end{equation} 
	We next consider the frequency regime \(2^k\leq t^{(1-4p_0)/5}\). Noticing that
	\begin{equation*}
	\begin{aligned}
	(2^{(1-\alpha)k/4}+2^{10k})\big\|\widehat{P_k^\prime f}\big\|_{L^\infty}\lesssim 1,\quad \|\widehat{P_k^\prime f}\|_{L^2}+2^k\|\partial\widehat{P_k^\prime f}\|_{L^2}\lesssim t^{p_0}, 
	\end{aligned}
	\end{equation*}
	and
	\begin{equation*}
	\begin{aligned}
	\frac{2^{(\frac{1-\alpha}{2}+\frac{3}{2})k}}{2^{(1-\alpha)k/4}+2^{10k}}\lesssim 1,\quad 2^{(-\frac{1+3\alpha}{4}+\frac{3}{2})k} \lesssim t^{\frac{1}{4}-p_0},
	\end{aligned}
	\end{equation*}   
	we then deduce from \eqref{l4} that
	\begin{equation*}
	\begin{aligned}
	2^{3k/2}\left\|P_ku\right\|_{L^\infty}
	&\lesssim t^{-\frac{1}{2}}2^{(\frac{1-\alpha}{2}+\frac{3}{2})k}\big\|\widehat{P_k^\prime f}\big\|_{L^\infty}\\
	&\quad+t^{-\frac{3}{4}}2^{(-\frac{1+3\alpha}{4}+\frac{3}{2})k}\big(\big\|\widehat{P_k^\prime f}\big\|_{L^2}+2^k\big\|\partial\widehat{P_k^\prime f}\big\|_{L^2}\big)\\
	&\lesssim t^{-\frac{1}{2}}.
	\end{aligned}
	\end{equation*}
	Hence we have shown
	\begin{equation}\label{l7}
	\begin{aligned}
	\quad 2^{3k/2}\left\|P_ku\right\|_{L^\infty}
	\lesssim t^{-\frac{1}{2}}.
	\end{aligned}
	\end{equation}

	To prove that  \(\|u\|_{L^\infty}\) satisfies \eqref{l6}, we first estimate the contribution of positive frequencies using \eqref{l7} : 
	\begin{equation*}
	\begin{aligned}
	\sum_{k\geq 0}\|P_ku\|_{L^\infty}\leq \sup_{k\geq 0}(2^{3 k/2}\left\|P_k u\right\|_{L^\infty})\sum_{k\geq 0}2^{-3k/2}
	\lesssim t^{-\frac{1}{2}} .
	\end{aligned}
	\end{equation*}
	For the negative frequencies, one first notices that  the contribution on  frequencies \(2^k\leq t^{-1}\) is obvious due to the following inequality 
	\begin{equation*}
	\begin{aligned}
	\|P_k u\|_{L^\infty}
	\lesssim 2^{k/2}\|u\|_{L^2}.
	\end{aligned}
	\end{equation*}
	Thus it remains to estimate the contribution of the frequencies \(t^{-1}\leq 2^k\leq 1\). To this aim, we split this range into two regimes \(t^{-\frac{4p_0}{1-\alpha}}\leq 2^k\leq 1 \) and \(t^{-1}\leq 2^k\leq t^{-\frac{4p_0}{1-\alpha}}\). For the latter case, it follows from \eqref{l5} that
	\begin{equation*}
	\begin{aligned}
	\sum_{t^{-1}\leq 2^k\leq t^{-\frac{4p_0}{1-\alpha}}}\left\|P_ku\right\|_{L^\infty}
	\lesssim t^{-\frac{1}{2}}t^{-2p_0}t^{p_0}\lesssim t^{-\frac{1}{2}}.
	\end{aligned}
	\end{equation*}
	In view of \eqref{l4}, the former case may be handled as follows:
	\begin{equation*}
	\begin{aligned}
	\sum_{t^{-\frac{4p_0}{1-\alpha}}\leq 2^k\leq 1}\left\|P_ku\right\|_{L^\infty}
	\lesssim t^{-\frac{1}{2}}\frac{2^{\frac{1-\alpha}{2}k}}{2^{(1-\alpha)k/4}+2^{10k}}
	+t^{-\frac{3}{4}}t^{\frac{(1+3\alpha)p_0}{(1-\alpha)}}t^{p_0}\lesssim t^{-\frac{1}{2}},
	\end{aligned}
	\end{equation*}
	in which we used the assumption \(p_0\in (0,10^{-3}(1-\alpha)]\).

	We finally prove the desired bound \eqref{l6} for \(\|\partial_xu\|_{L^\infty}\).  Indeed, by \eqref{l7}, one may estimate
	\begin{equation*}
	\begin{aligned}
	\|\partial_xu\|_{L^\infty}
	\lesssim \|u\|_{L^\infty}\sum_{k\leq 0}2^k
	+\sup_{k\geq 0}(2^{3 k/2}\left\|P_k u\right\|_{L^\infty})\sum_{k\geq 0}2^{-k/2}
	\lesssim t^{-\frac{1}{2}}.
	\end{aligned}
	\end{equation*}

\end{proof}

Lemma \ref{decay:fKdV} shows that \eqref{fKdV-decay} is a consequence of \eqref{fKdV2}.

\subsection{The main estimates}

We first have the following energy estimates: 
\begin{proposition}\label{lpr:1} Let \(u\) be a solution of the equation \eqref{eq:main-1} with initial data \eqref{eq:initial-1}  satisfying the a priori bounds \eqref{l8}. Then the following estimates hold true:
	\begin{align*}
	\|u(t,\cdot)\|_{H^{N_0}}\leq C\varepsilon_0\langle t \rangle^{C\varepsilon_1^2},
	\end{align*}
	and
	\begin{align*}
	\|f(t,\cdot)\|_{H^{1,1}}\leq C(\varepsilon_0+ \varepsilon_1^3)\langle t \rangle^{C\varepsilon_1^2}.
	\end{align*}	
\end{proposition}

\begin{proof} With the a priori bounds \eqref{l8}, the proof is identical to a similar one in  \cite[Theorem 3.1]{SW1}, thereby we omit it.
\end{proof}

To complete the proof of Theorem \ref{th:main-1},  it remains to show that \(\|f\|_Z\) is controlled by the bound on the right hand side of \eqref{l9}.  For such purpose, we begin by 
taking the Fourier transform of \eqref{eq:main-1} leading to
\begin{equation}\label{l10}
\begin{aligned}
\partial_t\widehat{f}(\xi,t)=-\frac{\mathrm{i}\xi}{6\pi}\int_{\mathbb{R}^2}e^{-\mathrm{i}t\Phi(\xi,\eta,\sigma)}
\widehat{f}(\xi-\eta-\sigma,t)\widehat{f}(\eta,t)\widehat{f}(\sigma,t)\, d \eta d \sigma
=\colon I(\xi,t),
\end{aligned}
\end{equation}
where the phase function \(\Phi\) takes the form
\begin{equation*}
\begin{aligned}
\Phi(\xi,\eta,\sigma)=|\xi|^{\alpha}\xi-|\xi-\eta-\sigma|^{\alpha}(\xi-\eta-\sigma)-|\eta|^{\alpha}\eta-|\sigma|^{\alpha}\sigma.
\end{aligned}
\end{equation*}
However the contribution of the space-time resonance in \(I(\xi,t)\) is not integrable, and one needs to remove it from the ODE \eqref{l10} by using an integrating factor which leads to the phase correction in the asymptotic behavior of the solutions. 
Hence, upon introducing the function
\begin{align}\label{l11}
H(\xi,t):=\frac{\xi|\xi|^{1-\alpha}}{|\alpha|(\alpha+1)}\int_1^t|\widehat{f}(\xi,s)|^2\,\frac{d s}{s},
\end{align}	
we define
\begin{align*}
g(\xi,t):=e^{\mathrm{i}H(\xi,t)}\widehat{f}(\xi,t),
\end{align*}
and we can rewrite the ODE \eqref{l10} as
\begin{align}\label{l12}
\partial_t g(\xi,t)=e^{\mathrm{i}H(\xi,t)}\big[ I(\xi,t)-\mathrm{i}\tilde{c}t^{-1}\xi|\xi|^{1-\alpha}|\widehat{f}(\xi,t)|^2\widehat{f}(\xi,t)\big],
\end{align}
here the constant \(\tilde{c}:=-[|\alpha|(\alpha+1)]^{-1}\).
To show the \(Z\)-norm bound in \eqref{l9}, it suffices to prove the following proposition: 
\begin{proposition}\label{lpr:2}
	It holds that
	\begin{align}\label{l13}
	t_1^{p_0}\left\|(|\xi|^{(1-\alpha)/4}+|\xi|^{10})\big[g(\xi,t_2)-g(\xi,t_1)\big]\right\|_{L^\infty_\xi}\lesssim\epsilon_0,
	\end{align}
	for any \(t_1\leq t_2\in[1,T]\).
\end{proposition}
It is easy to see that the estimate \eqref{fKdV3} is a consequence of \eqref{l13}.

\subsubsection{Reduction of \eqref{l13}}\label{reduction:fKdV}
In order to show \eqref{l13}, 
we split \(I(\xi,t)\) on the right hand side of \eqref{l10} in frequencies
\begin{equation}\label{l14}
\begin{aligned}
I(\xi,t)=-\mathrm{i}(6\pi)^{-1}\sum_{k_1,k_2,k_3\in \mathbb{Z}}I_{k_1,k_2,k_3}(\xi,t),
\end{aligned}
\end{equation}
in which each \(I_{k_1,k_2,k_3}(\xi,t)\) is given by
\begin{equation}\label{l15}
\begin{aligned}
I_{k_1,k_2,k_3}(\xi,t):=\xi\int_{\mathbb{R}^2}e^{-\mathrm{i}t\Phi(\xi,\eta,\sigma)}\widehat{f_{k_1}}(\xi-\eta-\sigma,t)\widehat{f_{k_2}}(\eta,t)\widehat{f_{k_3}}(\sigma,t)\, d \eta d \sigma.
\end{aligned}
\end{equation}

For the proof of \eqref{l13}, it suffices to establish that for \(m\in \{1,2,\dots\}\) it holds
\begin{align}\label{l16}
\left\|(|\xi|^{(1-\alpha)/4}+|\xi|^{10})\big[g(\xi,t_2)-g(\xi,t_1)\big]\right\|_{L^\infty_\xi}\lesssim\epsilon_02^{-p_0m},
\end{align}
for any \(t_1\leq t_2\in [2^m-1,2^{m+1}]\cap[1,T]\).

Let \(k\in\mathbb{Z}\) and \(m\in \{1,2,\dots\}\),  we always assume that \(|\xi|\in[2^k,2^{k+1}]\) and \(s\in[2^m-1,2^{m+1}]\cap[1,T]\) hereafter.
We first observe that for  \(k\in [p_0m,\infty)\cap(-\infty,-10p_0m/(1-\alpha)]\) there holds
\begin{equation}\label{l17}
\begin{aligned}
\big|(|\xi|^{(1-\alpha)/4}+|\xi|^{10})\widehat{f_k}(\xi)\big|\lesssim \epsilon_02^{-p_0m},
\end{aligned}
\end{equation}
which proves \eqref{l16} in this frequency regime. Indeed, via the interpolation estimate  \eqref{ap1}, the case \(k\in [p_0m,\infty)\) can be handled in a similar manner as \eqref{high frequency}, and for the other case \(k\in (-\infty,-10p_0m/(1-\alpha)]\) we estimate
\begin{equation*}
\begin{aligned}
&\big|(|\xi|^{(1-\alpha)/4}+|\xi|^{10})\widehat{f_k}(\xi)\big|
\lesssim 2^{(1-\alpha)k/4}\|f\|_{H^{1,1}}\\
&\lesssim \epsilon_02^{(1-\alpha)k/4}2^{p_0m}\lesssim \epsilon_02^{-p_0m}.
\end{aligned}
\end{equation*}

We recall the notation "positive part" and "negative part"
\begin{equation*}
\begin{aligned}
l_+=\max(l,0),\quad l_-=\max(-l,0),
\end{aligned}
\end{equation*}
and define the short-hand notation
\begin{equation}\label{notation:positive-negative}
l_{\pm}=
\begin{cases}
-10l_+,&\quad \text{if}\ l\geq 0,\\
(1-\alpha)l_-/4,&\quad \text{if}\ l< 0.
\end{cases}
\end{equation}
It follows from the a priori bounds \eqref{l8} and the localization that 
\begin{equation}\label{l18}
\begin{aligned}
&\|\widehat{f_l}(s)\|_{L^2}\lesssim \epsilon_12^{p_0m}\min(2^{-N_0l_+},2^{l/2}),\\
&\|\partial\widehat{f_l}(s)\|_{L^2}\lesssim \epsilon_12^{p_0m}\min\big[2^{-l},\max(2^{-l/2},1)\big],\\
&\|\widehat{f_l}(s)\|_{L^\infty}\lesssim \epsilon_12^{l_{\pm}},\\
&\|e^{s|D|^{\alpha}\partial_x}f_l(s)\|_{L^\infty}\lesssim \epsilon_12^{-m/2},
\end{aligned}
\end{equation}
for any  \(s\in[2^m-1,2^{m+1}]\cap[1,T]\) and any \(l\in\mathbb{Z}\).

According to the equation \eqref{l10} and the decomposition \eqref{l14}-\eqref{l15}, and taking \eqref{l17} into account, it suffices to complete the proof of \eqref{l16}  to prove
\begin{equation}\label{l19}
\begin{aligned}
&\sum_{k_1,k_2,k_3\in \mathbb{Z}}\bigg|\int_{t_1}^{t_2}e^{\mathrm{i}H(\xi,s)}\big[ I_{k_1,k_2,k_3}(\xi,s)-\tilde{\tilde{c}}s^{-1}\xi|\xi|^{1-\alpha}\widehat{f_{k_1}}(\xi,s)\widehat{f_{k_2}}(\xi,s)\widehat{f_{k_3}}(-\xi,s)\\
&\quad\quad\quad\quad\quad\quad\quad\quad\quad\quad\quad\quad\quad\quad\quad-\tilde{\tilde{c}}
s^{-1}\xi|\xi|^{1-\alpha}\widehat{f_{k_1}}(\xi,s)\widehat{f_{k_2}}(-\xi,s)\widehat{f_{k_3}}(\xi,s)\\
&\quad\quad\quad\quad\quad\quad\quad\quad\quad\quad\quad\quad\quad-\tilde{\tilde{c}}s^{-1}\xi|\xi|^{1-\alpha}\widehat{f_{k_1}}(-\xi,s)\widehat{f_{k_2}}(\xi,s)\widehat{f_{k_3}}(\xi,s)\big]\, d s\bigg|\\
&\quad\quad\quad\quad\quad\quad\quad\quad\quad\quad\quad\quad\quad\quad\quad\lesssim
\epsilon_1^32^{-p_0m}2^{-10k_+},
\end{aligned}
\end{equation}
for \(t_1\leq t_2\in [2^m-1,2^{m+1}]\cap [1,T]\) and \(|\xi|\in[2^k,2^{k+1}]\) with \(k\in [-10p_0m/(1-\alpha),p_0m]\), where \(\tilde{\tilde{c}}:=2\pi/[|\alpha|(\alpha+1)]\).

A further reduction on \eqref{l19} is possible, indeed
using the a priori bounds \eqref{l18}, one can easily check that
\begin{equation}\label{l20}
\begin{aligned}
|I_{k_1,k_2,k_3}(\xi,s)|&\lesssim \epsilon_1^32^{3p_0m}2^k2^{\min(k_1,k_2,k_3)/2}2^{-N_0({k_1}_++{k_2}_++{k_3}_+)},
\end{aligned}
\end{equation}
and
\begin{equation}\label{l21}
\begin{aligned}
|I_{k_1,k_2,k_3}(\xi,s)|\lesssim \epsilon_1^32^k2^{\min(k_1,k_2,k_3)}2^{\text{med}(k_1,k_2,k_3)}2^{[\max(k_1,k_2,k_3)]_{\pm}},
\end{aligned}
\end{equation}
and 
\begin{equation}\label{l22}
\begin{aligned}
&\quad s^{-1}|\xi|^{2-\alpha}\big(\big|\widehat{f_{k_1}}(\xi,s)\widehat{f_{k_2}}(\xi,s)\widehat{f_{k_3}}(-\xi,s)\big|+\big|\widehat{f_{k_1}}(\xi,s)\widehat{f_{k_2}}(-\xi,s)\widehat{f_{k_3}}(\xi,s)\big|\\
&\quad\quad\quad\quad\quad\quad\quad\quad\quad\quad\quad\quad\quad+\big|\widehat{f_{k_1}}(-\xi,s)\widehat{f_{k_2}}(\xi,s)\widehat{f_{k_3}}(\xi,s)\big|\big)\\
&\lesssim \epsilon_1^32^{-m}2^{(2-\alpha)k}2^{3k_{\pm}}{\bf{1}}_{[0,4]}\big(\max(|k_1-k|,|k_2-k|,|k_3-k|)\big),
\end{aligned}
\end{equation}
where \({\bf{1}}_A(x)\) is a characteristic function which equals \(1\) for \(x\in A\), and equals \(0\) otherwise.
With the bounds \eqref{l20}-\eqref{l22} at hand, it is straightforward to verify \eqref{l19} if one of the following conditions holds
\begin{equation*}
\begin{aligned}
&\min(k_1,k_2,k_3)\leq -4m,\\
&\max(k_1,k_2,k_3)\geq p_0m/10,\\
&\min(k_1,k_2,k_3)+\mathrm{med}(k_1,k_2,k_3)\leq -(1+10p_0)m.
\end{aligned}
\end{equation*}

On the other hand, one notices that \(-10p_0m/(1-\alpha)\geq -m/100\), where we used the assumptions \(p_0\in (0,10^{-3}(1-\alpha)]\) when \(0<\alpha<1\) and \(p_0\in (0,-10^{-3}\alpha]\) when \(-1<\alpha<0\).

Therefore,  we are now in a position to reduce \eqref{l19} to the following: 
\begin{proposition}\label{lpr:3}
	Assume that \(k,k_1,k_2,k_3\in\mathbb{Z}\), \(m\in\mathbb{Z}\cap[100,\infty)\), \(|\xi|\in[2^k,2^{k+1}]\), and \(t_1\leq t_2\in [2^m-1,2^{m+1}]\cap [1,T]\). If
	\begin{equation}\label{l23}
	\begin{aligned}
	&k\in [-m/100,p_0m],\\
	&k_1,k_2,k_3\in[-4m,p_0m/10],\\
	&\min(k_1,k_2,k_3)+\mathrm{med}(k_1,k_2,k_3)\geq -(1+10p_0)m,
	\end{aligned}
	\end{equation}
	then
	\begin{equation}\label{l24}
	\begin{aligned}
	&\bigg|\int_{t_1}^{t_2}e^{\mathrm{i}H(\xi,s)}\big[ I_{k_1,k_2,k_3}(\xi,s)-\tilde{\tilde{c}}s^{-1}\xi|\xi|^{1-\alpha}\widehat{f_{k_1}}(\xi,s)\widehat{f_{k_2}}(\xi,s)\widehat{f_{k_3}}(-\xi,s)\\
	&\quad\quad\quad\quad\quad\quad\quad\quad\quad\quad\quad-\tilde{\tilde{c}}
	s^{-1}\xi|\xi|^{1-\alpha}\widehat{f_{k_1}}(\xi,s)\widehat{f_{k_2}}(-\xi,s)\widehat{f_{k_3}}(\xi,s)\\
	&\quad\quad\quad\quad\quad\quad\quad\quad\quad\quad\quad-\tilde{\tilde{c}}s^{-1}\xi|\xi|^{1-\alpha}\widehat{f_{k_1}}(-\xi,s)\widehat{f_{k_2}}(\xi,s)\widehat{f_{k_3}}(\xi,s)\big]\, d s\bigg|\\
	&\quad\quad\quad\quad\quad\quad\quad\quad\quad\quad\quad\lesssim\epsilon_1^32^{-2p_0m}2^{-10k_+}.
	\end{aligned}
	\end{equation}
\end{proposition}

\subsubsection{Proof of Proposition \ref{lpr:3}}
To prove \eqref{l24}, we shall split the frequencies in a more careful way, namely splitting the frequencies into a non-resonant  and a resonant regime, in which different techniques will be performed respectively.   

\bigskip
\noindent{\bf{Non-resonant regime: \(\max(|k_1-k|,|k_2-k|,|k_3-k|)\geq 21\).}} 
\begin{lemma}  If
	\begin{equation}\label{l25}
	\begin{aligned}
	\min(k_1,k_2,k_3)\geq -(1-20p_0)m,
	\end{aligned}
	\end{equation}
	then the bound \eqref{l24} holds.
\end{lemma}

\begin{proof}  We are done if we show
	\begin{equation*}
	\begin{aligned}
	| I_{k_1,k_2,k_3}(\xi,s)|
	\lesssim \epsilon_1^32^{-m}2^{-3p_0m}2^{-10k_+}.
	\end{aligned}
	\end{equation*}
	\noindent{\bf{{Case 1:} \(\max(|k_1-k_2|,|k_1-k_3|,|k_2-k_3|)\geq 5\).}}
	By symmetry, we may assume that \(|k_1-k_2|\geq5\) and \(\max(k_1,k_2)\geq k-20\). 	
	We first consider the case of \(\alpha\in(0,1)\). 
	One integrates by parts in \(\eta\) to find that
	\begin{align}\label{l26}
	|I_{k_1,k_2,k_3}(\xi,s)|\leq |F_1(\xi,s)|+|F_2(\xi,s)|+|F_3(\xi,s)|,
	\end{align}
	where each term on the right hand side is given by
	\begin{equation}\label{l27}
	\begin{aligned}
	F_1(\xi,s)&=\xi\int_{\mathbb{R}^2}m_1(\eta,\sigma)e^{-\mathrm{i}s\Phi(\xi,\eta,\sigma)}\partial_\eta\widehat{f_{k_1}}(\xi-\eta-\sigma,s)\widehat{f_{k_2}}(\eta,s)\widehat{f_{k_3}}(\sigma,s) \, d \eta d \sigma,\\
	F_2(\xi,s)&=\xi\int_{\mathbb{R}^2}m_1(\eta,\sigma)e^{-\mathrm{i}s\Phi(\xi,\eta,\sigma)}\widehat{f_{k_1}}(\xi-\eta-\sigma,s)\partial_\eta\widehat{f_{k_2}}(\eta,s)\widehat{f_{k_3}}(\sigma,s) \, d \eta d \sigma,\\
	F_3(\xi,s)&=\xi\int_{\mathbb{R}^2}\partial_\eta m_1(\eta,\sigma)e^{-\mathrm{i}s\Phi(\xi,\eta,\sigma)}\widehat{f_{k_1}}(\xi-\eta-\sigma,s)\widehat{f_{k_2}}(\eta,s)\widehat{f_{k_3}}(\sigma,s)\, d \eta d \sigma,
	\end{aligned}
	\end{equation}
	in which the multiplier takes the form
	\begin{align*}
	m_1(\eta,\sigma):=\frac{1}{s\partial_\eta\Phi(\xi,\eta,\sigma)}\varphi_{k_1}^\prime(\xi-\eta-\sigma)\varphi_{k_2}^\prime(\eta)\varphi_{k_3}^\prime(\sigma).
	\end{align*}
	
	In view of the support properties of the integral, we observe that
	\begin{equation*}
	\begin{aligned}
	|\partial_\eta\Phi(\xi,\eta,\sigma)|
	=(\alpha+1)\big||\xi-\eta-\sigma|^{\alpha}-|\eta|^{\alpha}\big|
	\gtrsim 2^{\alpha\max(k_1,k_2)}\gtrsim 2^{\alpha k},
	\end{aligned}
	\end{equation*}
	and hence have 
	\begin{equation}\label{l28}
	\begin{aligned}
	\|\mathcal{F}^{-1}(m_1)\|_{L^1}\lesssim 2^{-m}2^{-\alpha k},
	\end{aligned}
	\end{equation}
	and
	\begin{equation}\label{l29}
	\begin{aligned}
	\|\mathcal{F}^{-1}(\partial_\eta m_1)\|_{L^1}\lesssim 2^{-m}2^{-\alpha k}2^{-\min(k_1,k_2)}.
	\end{aligned}
	\end{equation}	
	Fixing \(\xi\) and \(s\), and letting
	\begin{equation}\label{l30}
	\begin{aligned}
	&\widehat{g_1}(\theta):=e^{\mathrm{i}s|\xi-\theta|^\alpha(\xi-\theta)}\partial_\eta\widehat{f_{k_1}}(\xi-\theta,s),\\
	&\widehat{g_2}(\eta):=e^{\mathrm{i}s|\eta|^\alpha\eta}\widehat{f_{k_2}}(\eta,s),\\
	&\widehat{g_3}(\sigma):=e^{\mathrm{i}s|\sigma|^\alpha\sigma}\widehat{f_{k_3}}(\sigma,s).
	\end{aligned}
	\end{equation}
	one uses \eqref{l18} to get
	\begin{equation*}
	\begin{aligned}
	\|g_1\|_{L^2}\lesssim \epsilon_12^{p_0m}{\max(2^{-k_1/2},1)}
	,\ \|g_2\|_{L^2}
	\lesssim \epsilon_12^{p_0m}2^{-N_0{k_2}_+},\
	\|g_3\|_{L^\infty}\lesssim \epsilon_12^{-m/2}.
	\end{aligned}
	\end{equation*}
	This together with \eqref{l28} and \eqref{ap2} implies  
	\begin{align}\label{estimate:F_1}
	|F_1(\xi,s)|
	\lesssim\epsilon_1^32^{(1-\alpha)k}\max(2^{-k_1/2},1)2^{(2p_0-3/2)m}2^{-N_0\max({k_2}_+,{k_3}_+)}.
	\end{align}
	Analogously it holds that
	\begin{align}\label{estimate:F_2}
	|F_2(\xi,s)|\lesssim\epsilon_1^32^{(1-\alpha)k}\max(2^{-k_2/2},1)2^{(2p_0-3/2)m}2^{-N_0\max({k_1}_+,{k_3}_+)}.
	\end{align}	
	For fixed \(\xi\) and \(s\), we let
	\begin{equation}\label{l31}
	\begin{aligned}
	&\widehat{g_1}(\theta):=e^{\mathrm{i}s|\xi-\theta|^\alpha(\xi-\theta)}
	\widehat{f_{k_1}}(\xi-\theta,s),\\
	&\widehat{g_2}(\eta):=e^{\mathrm{i}s|\eta|^\alpha\eta}\widehat{f_{k_2}}(\eta,s),\\
	&\widehat{g_3}(\sigma):=e^{\mathrm{i}s|\sigma|^\alpha\sigma}\widehat{f_{k_3}}(\sigma,s),
	\end{aligned}
	\end{equation}
	and use \eqref{l18} to estimate
	\begin{equation*}
	\begin{aligned}
	\|g_1\|_{L^2}\lesssim \epsilon_12^{p_0m}2^{k_1/2}
	,\ 
	\|g_2\|_{L^\infty}\lesssim \epsilon_12^{-m/2},\
	\|g_3\|_{L^2}
	\lesssim \epsilon_12^{p_0m}2^{-N_0{k_3}_+},
	\end{aligned}
	\end{equation*}
	which combines \eqref{l29} and \eqref{ap2} lead to 
	\begin{align}\label{estimate:F_3}
	|F_3(\xi,s)|
	\lesssim \epsilon_1^32^{(1-\alpha)k}2^{-\min(k_1,k_2)/2}2^{(2p_0-3/2)m}2^{-N_0{k_3}_+}.
	\end{align}
	It follows from \eqref{estimate:F_1}, \eqref{estimate:F_2} and \eqref{estimate:F_3} that
	\begin{equation*}
	\begin{aligned}
	&\quad|F_1(\xi,s)|+|F_2(\xi,s)|+|F_3(\xi,s)|\\
	&\lesssim \epsilon_1^32^{(1-\alpha)k}\max(2^{-\min(k_1,k_2)/2},1)2^{(2p_0-3/2)m}2^{-10k_+}
	(2^{10\max({k_1}_+,{k_2}_+,{k_3}_+)}+1)\\
	&\lesssim\epsilon_1^32^{-m}2^{-5p_0m}2^{-10k_+},
	\end{aligned}
	\end{equation*}
	where in the last line we used \eqref{l23} and \eqref{l25}.
	
	We next consider the case of \(\alpha\in(-1,0)\). On the support of the integral, one instead has 
	\begin{equation*}
	\begin{aligned}
	|\partial_\eta\Phi(\xi,\eta,\sigma)|
	\gtrsim 2^{\alpha\min(k_1,k_2)}.
	\end{aligned}
	\end{equation*}
	Hence it holds that
	\begin{equation*}
	\begin{aligned}
	\|\mathcal{F}^{-1}(m_5)\|_{L^1}\lesssim 2^{-m}2^{-\alpha\min(k_1,k_2)},
	\end{aligned}
	\end{equation*}
	and
	\begin{equation*}
	\begin{aligned}
	\|\mathcal{F}^{-1}(\partial_\eta m_5)\|_{L^1}\lesssim 2^{-m}2^{-\alpha \min(k_1,k_2)}2^{-\min(k_1,k_2)}.
	\end{aligned}
	\end{equation*}
	In virtue of \(\alpha\in(-1,0)\) and \eqref{l23}, the term \(2^{-\alpha\min(k_1,k_2)}\) is not so harmful, which contributes an acceptable bound \(2^{p_0m/10}\). Then one may repeat the argument for the case of \(\alpha\in(0,1)\) to obtain the desired result.  	
	
	\bigskip
	\noindent{\bf{Case 2: \(\max(|k_1-k_2|,|k_1-k_3|,|k_2-k_3|)\leq 4\).}}
	In this case  \(\partial_\eta\Phi\neq0\) or \(\partial_\sigma\Phi\neq0\), and we will assume \(\partial_\eta\Phi\neq0\) without loss of generality and then perform the same operation as
	\eqref{l26}-\eqref{l27}. 
	Recalling \(2^{k_1}\approx 2^{k_2}\approx 2^{k_3}\), on the support of the integral,  we have
	\begin{equation*}
	\begin{aligned}
	|\partial_\eta\Phi(\xi,\eta,\sigma)|
	=(\alpha+1)\big||\xi-\eta-\sigma|^{\alpha}-|\eta|^{\alpha}\big|
	\gtrsim 2^{\alpha k_2}.
	\end{aligned}
	\end{equation*}
	The bound yields
	\begin{equation}\label{l32}
	\begin{aligned}
	\|\mathcal{F}^{-1}(m_1)\|_{L^1}\lesssim 2^{-m}2^{-\alpha k_2},
	\end{aligned}
	\end{equation}
	and
	\begin{equation}\label{l33}
	\begin{aligned}
	\|\mathcal{F}^{-1}(\partial_\eta m_1)\|_{L^1}\lesssim 2^{-m}2^{-(\alpha+1)k_2}.
	\end{aligned}
	\end{equation}
	Define 	
	\(g_1,g_2\) and \(g_3\) as \eqref{l30},
	it then follows from \eqref{l18} that
	\begin{equation*}
	\begin{aligned}
	\|g_1\|_{L^2}\lesssim \epsilon_12^{p_0m}2^{-k_1}
	,\ \|g_2\|_{L^2}
	\lesssim \epsilon_12^{p_0m}2^{-N_0{k_2}_+},\
	\|g_3\|_{L^\infty}\lesssim \epsilon_12^{-m/2},
	\end{aligned}
	\end{equation*}
	which uses \eqref{l32} by applying \eqref{ap2} to give
	\begin{align}\label{estimate:F_1-2}
	|F_1(\xi,s)|
	\lesssim\epsilon_1^32^k2^{-(\alpha+1)k_2}2^{(2p_0-3/2)m}2^{-N_0{k_2}_+},
	\end{align}
	where we used \(2^{k_1}\approx 2^{k_2}\approx 2^{k_3}\).
	In a similar manner one may obtain
	\begin{align}\label{estimate:F_2-2}
	|F_2(\xi,s)|\lesssim\epsilon_1^32^k2^{-(\alpha+1)k_2}2^{(2p_0-3/2)m}2^{-N_0{k_2}_+}.
	\end{align}
	Recalling the definition of \(g_1,g_2\) and \(g_3\) in \eqref{l31},
	one invokes \eqref{l18} to see
	\begin{equation*}
	\begin{aligned}
	\|g_1\|_{L^2}\lesssim \epsilon_12^{p_0m}
	,\ \|g_2\|_{L^2}
	\lesssim \epsilon_12^{p_0m}2^{-N_0{k_2}_+},\
	\|g_3\|_{L^\infty}\lesssim \epsilon_12^{-m/2},
	\end{aligned}
	\end{equation*}
	which combining  \eqref{l33} and applying \eqref{ap2}  yield 
	\begin{align}\label{estimate:F_3-2}
	|F_3(\xi,s)|
	\lesssim \epsilon_1^32^k2^{-(\alpha+1)k_2}2^{(2p_0-3/2)m}2^{-N_0{k_2}_+}.
	\end{align}
	We conclude \eqref{estimate:F_1-2}-\eqref{estimate:F_3-2} that
	\begin{equation*}
	\begin{aligned}
	|F_1(\xi,s)|+|F_2(\xi,s)|+|F_3(\xi,s)|
	&\lesssim\epsilon_1^32^k2^{-(\alpha+1)k_2}2^{(2p_0-3/2)m}2^{-N_0{k_2}_+}\\
	&\lesssim \epsilon_1^32^{-m}2^{-5p_0m}2^{-10k_+},
	\end{aligned}
	\end{equation*}
	in which \(-(\alpha+1)k_2\leq -(\alpha+1)k+14(\alpha+1)\) and \eqref{l23} were used.

\end{proof}

It remains to consider the case  \(\min(k_1,k_2,k_3)\leq -(1-20p_0)m\). This case is more delicate and necessitates  a more careful analysis of the resonances. To this end,
without loss of generality, we may assume that \(\xi>0\) and \(\xi\in[2^k,2^{k+1}]\), and break up the integral \(I_{k_1,k_2,k_3}\) as follows:
\begin{align*}
I_{k_1,k_2,k_3}(\xi,s)=\sum_{\iota_1,\iota_2,\iota_3\in\{+,-\}}I_{k_1,k_2,k_3}^{\iota_1,\iota_2,\iota_3}(\xi,s),
\end{align*}
where each component is defined by
\begin{align*}
I_{k_1,k_2,k_3}^{\iota_1,\iota_2,\iota_3}(\xi,s)=\xi\int_{\mathbb{R}^2}e^{-\mathrm{i}t\Phi(\xi,\eta,\sigma)}\widehat{f_{k_1}^{\iota_1}}(\xi-\eta-\sigma,s)\widehat{f_{k_2}^{\iota_2}}(\eta,s)\widehat{f_{k_3}^{\iota_3}}(\sigma,s)\, d\eta d\sigma,
\end{align*}
with the notation \(\widehat{f_{l}^{\iota}}(\mu):=\widehat{f_{l}}(\mu)1_{\iota}(\mu),1_+:=1_{[0,\infty)},1_-:=1_{(-\infty,0]}\). First observing that \(I_{k_1,k_2,k_3}^{-,-,-}(\xi,s)=0\), we are reduced  to consider the other seven cases:\\
\[\{(+,+,+),(+,-,-),(-,+,-),(-,-,+),(+,+,-),(+,-,+),(-,+,+)\}.\] 

\begin{lemma} If
	\begin{equation}\label{l34}
	\begin{aligned}
	\min(k_1,k_2,k_3)\leq -(1-20p_0)m,
	\end{aligned}
	\end{equation}
	then the bound \eqref{l24} holds.
\end{lemma}
\begin{proof} It suffices to show
	\begin{align*}
	\bigg|\int_{t_1}^{t_2}e^{\mathrm{i}H(\xi,s)}I_{k_1,k_2,k_3}^{\iota_1,\iota_2,\iota_3}(\xi,s)\, d s\bigg|
	\lesssim \varepsilon_1^32^{-2p_0m}2^{-10k_+},
	\end{align*}
	and we need to consider several cases.
	
	\bigskip
	\noindent{\bf{Case 1: \((\iota_1,\iota_2,\iota_3)= (+,+,+)\).}} In this case,  
	one calculates that
	\begin{equation}\label{l35}
	\begin{aligned}
	\Phi(\xi,\eta,\sigma)&=\xi^{\alpha+1}-(\xi-\eta-\sigma)^{\alpha+1}-\eta^{\alpha+1}-\sigma^{\alpha+1}\\
	&\gtrsim 2^{(\alpha+1)\mathrm{med}(k_1,k_2,k_3)}
	\gtrsim 2^{-30(\alpha+1)p_0m},
	\end{aligned}
	\end{equation}
	where in the second inequality we used \eqref{l23} and \eqref{l34}, and in the first inequality we used the inequalities 
	\begin{align}\label{key inequality-1}
	a^{\alpha+1}+b^{\alpha+1}+c^{\alpha+1}-(a+b+c)^{\alpha+1}\gtrsim ba^\alpha,\quad \text{if}\ \alpha\in(0,1),
	\end{align}
	and
	\begin{align}\label{key inequality-2}
	a^{\alpha+1}+b^{\alpha+1}+c^{\alpha+1}-(a+b+c)^{\alpha+1}\gtrsim a^{\alpha+1},\quad \text{if}\ \alpha\in(-1,0),
	\end{align}
	for \(a\geq b\geq c \in(0,\infty)\).

	We integrate by parts in \(s\) to find that
	\begin{equation}\label{l36}
	\begin{aligned}
	&\quad\bigg|\xi\int_{t_1}^{t_2}\int_{\mathbb{R}^2}e^{\mathrm{i}H(\xi,s)}e^{-\mathrm{i}s\Phi(\xi,\eta,\sigma)}\widehat{f_{k_1}^{\iota_1}}(\xi-\eta-\sigma,s)\widehat{f_{k_2}^{\iota_2}}(\eta,s)\widehat{f_{k_3}^{\iota_3}}(\sigma,s)
	\, d \eta d \sigma d s\bigg|\\
	&\lesssim  
	\sum_{j=1}^2\bigg|\xi\int_{\mathbb{R}^2}\frac{e^{\mathrm{i}H(\xi,t_j)}e^{-\mathrm{i}t_j\Phi(\xi,\eta,\sigma)}}{\Phi(\xi,\eta,\sigma)}\widehat{f_{k_1}
		^{\iota_1}}(\xi-\eta-\sigma,t_j)\widehat{f_{k_2}^{\iota_2}}(\eta,t_j)\widehat{f_{k_3}^{\iota_3}}(\sigma,t_j)\, d \eta d \sigma\bigg|\\
	&\quad+\int_{t_1}^{t_2}\bigg|\xi\int_{\mathbb{R}^2}\frac{e^{-\mathrm{i}s\Phi(\xi,\eta,\sigma)}}{\Phi(\xi,\eta,\sigma)}\frac{d}{d s}\bigg[e^{\mathrm{i}H(\xi,s)}\widehat{f_{k_1}^{\iota_1}}(\xi-\eta-\sigma,s)\widehat{f_{k_2}^{\iota_2}}(\eta,s)\widehat{f_{k_3}^{\iota_3}}(\sigma,s)\bigg]\, d \eta d \sigma \bigg| d s\\
	&=\colon \sum_{j=1}^2A_j(\xi,t_j)+B(\xi).
	\end{aligned}
	\end{equation}
	To estimate \(A_j(\xi,t_j),\ j=1,2\), we define
	\begin{equation*}
	\begin{aligned}
	m_2(\eta,\sigma):=\frac{1}{\Phi(\xi,\eta,\sigma)}\varphi_{k_1}^\prime(\xi-\eta-\sigma)\varphi_{k_2}^\prime(\eta)\varphi_{k_3}^\prime(\sigma),
	\end{aligned}
	\end{equation*}
	and thus use \eqref{l35} to deduce
	\begin{equation*}
	\begin{aligned}
	\|\mathcal{F}^{-1}(m_2)\|_{L^1}\lesssim 2^{30(\alpha+1)p_0m}.
	\end{aligned}
	\end{equation*}
	This estimate together with the inequality \eqref{ap2} allows  to bound
	\begin{equation*}
	\begin{aligned}
	A_j(\xi,t_j)
	&\lesssim \epsilon_1^32^k2^{30(\alpha+1)p_0m}2^{(2p_0-1/2)m}2^{-N_0({k_1}_+,{k_2}_+,{k_3}_+)}\\
	&\lesssim \epsilon_1^32^{-m/4}2^{-10k_+},\quad \text{for}\ j=1,2,
	\end{aligned}
	\end{equation*}	
	
	We now turn to estimating the term \(B(\xi)\). For this, we expand \(d /d s\) to see that		
	\begin{equation*}
	\begin{aligned}
	B(\xi)\lesssim 2^m\sup_{s\in[t_1,t_2]}\big[B_1(\xi,s)+B_2(\xi,s)+B_3(\xi,s)+B_4(\xi,s)\big],
	\end{aligned}
	\end{equation*}
	where each term on the right hand side is given by
	\begin{equation}\label{l37}
	\begin{aligned}
	B_1(\xi,s)&= \bigg|\xi\int_{\mathbb{R}^2}m_2(\eta,\sigma)e^{-\mathrm{i}s\Phi(\xi,\eta,\sigma)}\partial_s\widehat{f_{k_1}^{\iota_1}}(\xi-\eta-\sigma,s)\widehat{f_{k_2}^{\iota_2}}(\eta,s)\widehat{f_{k_3}^{\iota_3}}(\sigma,s)\,
	d \eta d \sigma\bigg|,\\
	B_2(\xi,s)&= \bigg|\xi\int_{\mathbb{R}^2}m_2(\eta,\sigma)e^{-\mathrm{i}s\Phi(\xi,\eta,\sigma)}\widehat{f_{k_1}^{\iota_1}}(\xi-\eta-\sigma,s)\partial_s\widehat{f_{k_2}^{\iota_2}}(\eta,s)\widehat{f_{k_3}^{\iota_3}}(\sigma,s)\, d \eta d \sigma\bigg|,\\
	B_3(\xi,s)&= \bigg|\xi\int_{\mathbb{R}^2}m_2(\eta,\sigma)e^{-\mathrm{i}s\Phi(\xi,\eta,\sigma)}\widehat{f_{k_1}^{\iota_1}}(\xi-\eta-\sigma,s)\widehat{f_{k_2}^{\iota_2}}(\eta,s)\partial_s\widehat{f_{k_3}^{\iota_3}}(\sigma,s)\, d \eta d \sigma\bigg|,\\
	B_4(\xi,s)&=\bigg|\xi\int_{\mathbb{R}^2}m_3(\eta,\sigma)e^{-\mathrm{i}s\Phi(\xi,\eta,\sigma)}\widehat{f_{k_1}^{\iota_1}}(\xi-\eta-\sigma,s)\widehat{f_{k_2}^{\iota_2}}(\eta,s)\widehat{f_{k_3}^{\iota_3}}(\sigma,s)\,
	d \eta d \sigma\bigg|,
	\end{aligned}
	\end{equation}
	with the multiplier
	\begin{equation*}
	\begin{aligned}
	m_3(\eta,\sigma):=\partial_s H(\xi,s)m_2(\eta,\sigma).
	\end{aligned}
	\end{equation*}
	Recalling the definition \eqref{l11}, one may estimate 	
	\begin{equation*}
	\begin{aligned}
	|\partial_s H(\xi,s)|
	\lesssim \epsilon_1^22^{(2-\alpha)k}2^{-m}2^{2k_{\pm}},
	\end{aligned}
	\end{equation*}
	and then combine with \eqref{l35} to deduce
	\begin{equation*}
	\begin{aligned}
	\|\mathcal{F}^{-1}(m_3)\|_{L^1}\lesssim \epsilon_1^22^{(2-\alpha)k}2^{30(\alpha+1)p_0m}2^{-m}2^{2k_{\pm}}.
	\end{aligned}
	\end{equation*}
	Using this resulting inequality and	applying \eqref{ap2}, we can obtain
	\begin{equation*}
	\begin{aligned}
	\sup_{s\in[t_1,t_2]}B_4(\xi,s)
	&\lesssim \epsilon_1^52^{(3-\alpha)k}2^{30(\alpha+1)p_0m}2^{(2p_0-3/2)m}2^{-N_0({k_1}_+,{k_2}_+,{k_3}_+)}
	2^{2k_{\pm}}\\
	&\lesssim \epsilon_1^52^{-m}2^{-m/4}2^{-10k_+}.
	\end{aligned}
	\end{equation*}		
	The terms \(B_j(\xi,s),\ j=1,2,3\) are the same type and can be handled analogously. To this end, we would like to estimate \(\partial_s\widehat{f_l}(s)\) in \(L^2\)-norm. Indeed, in view of \eqref{l10} and \eqref{l18}, it is straightforward to see that 
	\begin{align}\label{l38}
	\|\partial_s\widehat{f_l}(s)\|_{L^2}\lesssim \epsilon_1^32^{-m}2^{3p_0m}2^{-10l_+}.
	\end{align}
	Again fix \(\xi\) and \(s\), and let
	\begin{equation*}
	\begin{aligned}
	&\widehat{g_1}(\theta):=e^{\mathrm{i}s|\xi-\theta|^\alpha(\xi-\theta)}\partial_s\widehat{f_{k_1}^{\iota_1}}(\xi-\theta,s),\\
	&\widehat{g_2}(\eta):=e^{\mathrm{i}s|\eta|^\alpha\eta}\widehat{f_{k_2}^{\iota_2}}(\eta,s),\\
	&\widehat{g_3}(\sigma):=e^{\mathrm{i}s|\sigma|^\alpha\sigma}\widehat{f_{k_3}^{\iota_3}}(\sigma,s),
	\end{aligned}
	\end{equation*}
	we then use \eqref{l18} and \eqref{l38} to get
	\begin{equation*}
	\begin{aligned}
	\|g_1\|_{L^2}\lesssim \epsilon_12^{-m}2^{3p_0m}2^{-10{k_1}_{+}},\ \|g_2\|_{L^2}
	\lesssim \epsilon_12^{p_0m}2^{-N_0{k_2}_+},\
	\|g_3\|_{L^\infty}\lesssim \epsilon_12^{-m/2}.
	\end{aligned}
	\end{equation*}
	Hence it follows from \eqref{ap2} that	
	\begin{equation*}
	\begin{aligned}
	\sup_{s\in[t_1,t_2]}B_1(\xi,s)
	&\lesssim \epsilon_1^32^k2^{30(\alpha+1)p_0m}2^{(4p_0-3/2)m}2^{-10{k_1}_{+}}2^{-N_0({k_2}_+,{k_3}_+)}\\
	&\lesssim \epsilon_1^32^{-m}2^{-m/4}2^{-10k_+},
	\end{aligned}
	\end{equation*}
	Proceeding  as above,  one may estimate
	\begin{equation*}
	\begin{aligned}
	\sup_{s\in[t_1,t_2]}\big[B_2(\xi,s)+B_3(\xi,s)\big]
	\lesssim \epsilon_1^32^{-m}2^{-m/4}2^{-10k_+}.
	\end{aligned}
	\end{equation*}

	\bigskip
	\noindent{\bf{Case 2: \((\iota_1,\iota_2,\iota_3)\in \{(+,-,-),(-,+,-),(-,-,+)\}\).}}
	These three cases can be handled in a similar manner, thus we only focus on \((\iota_1,\iota_2,\iota_3)= (+,-,-)\).  
	In this case, we have
	\begin{equation*}
	\begin{aligned}
	-\Phi(\xi,\eta,\sigma)
	&=(\xi-\eta-\sigma)^{\alpha+1}-\xi^{\alpha+1}-(-\eta)^{\alpha+1}-(-\sigma)^{\alpha+1}\\
	&\gtrsim 2^{(\alpha+1)\mathrm{med}(k,k_2,k_3)}.
	\end{aligned}
	\end{equation*}
	
	We claim that
	\begin{equation}\label{l39}
	\begin{aligned}
	-\Phi(\xi,\eta,\sigma)
	\gtrsim \min(2^{-30(\alpha+1)p_0m},2^{-(\alpha+1)m/100}).
	\end{aligned}
	\end{equation}
	To verify this claim, one shall consider the following three sub-cases:\\
	{\bf{(i)}} \(\mathrm{med}(k,k_2,k_3)=k\). It is immediate that  
	\begin{equation*}
	\begin{aligned}
	-\Phi(\xi,\eta,\sigma)\gtrsim 2^{(\alpha+1)k}\gtrsim 2^{-(\alpha+1)m/100}.
	\end{aligned}
	\end{equation*}
	{\bf{(ii)}} \(\mathrm{med}(k,k_2,k_3)=k_2\).  
	If \(k\leq k_2\leq k_3 \), then one has
	\begin{equation*}
	\begin{aligned}
	-\Phi(\xi,\eta,\sigma)\gtrsim 2^{(\alpha+1)k_2}\geq 2^{(\alpha+1)k}\gtrsim 2^{-(\alpha+1)m/100}.
	\end{aligned}
	\end{equation*}
	If \(k_3\leq k_2\leq k\), then it holds
	\begin{equation*}
	\begin{aligned}
	-\Phi(\xi,\eta,\sigma)\gtrsim 2^{(\alpha+1)k_2}\geq 2^{(\alpha+1)\mathrm{med}(k_1,k_2,k_3)}\gtrsim 2^{-30(\alpha+1)p_0m}.
	\end{aligned}
	\end{equation*}
	{\bf{(iii)}} \(\mathrm{med}(k,k_2,k_3)=k_3\). This case is   symmetric case  {\bf{(ii)}}.

	Now the end of the poof proceeds similarly to {\bf{Case 1}} by using the bound \eqref{l39} and applying the integration by parts \eqref{l36}-\eqref{l37}. 
	
	\bigskip
	\noindent{\bf{Case 3: \((\iota_1,\iota_2,\iota_3)\in \{(+,+,-),(+,-,+),(-,+,+)\}\).}} We only consider the case \((\iota_1,\iota_2,\iota_3)=(+,+,-)\), and the other cases may be treated in a similar fashion.
	In this case, one has
	\begin{equation*}
	\begin{aligned}
	\Phi(\xi,\eta,\sigma)=\xi^{\alpha+1}-(\xi-\eta-\sigma)^{\alpha+1}-\eta^{\alpha+1}+(-\sigma)^{\alpha+1}.
	\end{aligned}
	\end{equation*}
	
	We consider the following two sub-cases.\\
	{\bf{(i) \(k_3=\min(k_1,k_2,k_3)\).}} In this case,  it follows from \eqref{l23} that
	\[
	k_3\in [-(1+20p_0)m,-(1-20p_0)m],\quad k_1,k_2\in [-30p_0m,p_0m/10].
	\]
	We have 
	\begin{equation}\label{l40}
	\begin{aligned}
	-\Phi(\xi,\eta,\sigma)
	&=\big[-\xi^{\alpha+1}+(\xi-\eta-\sigma)^{\alpha+1}
	+(\eta+\sigma)^{\alpha+1}\big]\\
	&\quad+\big[\eta^{\alpha+1}-(\eta+\sigma)^{\alpha+1}\big]
	-(-\sigma)^{\alpha+1}\\
	&\gtrsim 2^{(\alpha+1)\min(k_1,k_2)}\geq 2^{-30(\alpha+1)p_0m},
	\end{aligned}
	\end{equation}
	on the support of the integral.

	\vspace*{4pt}\noindent {\bf{(ii) \(k_3\neq \min(k_1,k_2,k_3)\).}} 
	In this case, using \eqref{l23} again, we have
	\[
	k_1\in [-(1+20p_0)m,-(1-20p_0)m],\quad k_2,k_3\in [-30p_0m,p_0m/10].
	\]	
	Observing \(k\gg k_1\), and taking the properties of support of the integral into account, one may estimate 
	\begin{equation}\label{l41}
	\begin{aligned}
	\quad\Phi(\xi,\eta,\sigma)
	&\geq \big[\xi^{\alpha+1}+(-\sigma)^{\alpha+1}-(\xi-\sigma)^{\alpha+1}\big]-(\xi-\eta-\sigma)^{\alpha+1}\\
	&\quad-|\eta^{\alpha+1}-(\xi-\sigma)^{\alpha+1}|\\
	&\geq 2^{(\alpha+1)\min(k,k_3)-1}-2^{(\alpha+1)k_1+1}-2^{k_1+\alpha k_2+10}\\
	&\geq 2^{(\alpha+1)\min(k,k_3)-2}
	\gtrsim \min(2^{-30(\alpha+1)p_0m},2^{-(\alpha+1)m/100}),
	\end{aligned}
	\end{equation}
	where we used \eqref{l23} in the last inequality.

	With \eqref{l40} and \eqref{l41} at hand,  we conclude the proof as in {\bf{Case 1}}.
\end{proof}

\bigskip
\noindent{\bf{Resonant regime: \(k_1,k_2,k_3\in[k-20,k+20]\).}} 

In oder to show \eqref{l24}, it is enough to prove the following two lemmas.

\begin{lemma}\label{le:3}
	We have
	\begin{align}\label{32}
	\bigg|\int_{t_1}^{t_2}e^{\mathrm{i}H(\xi,s)}I_{k_1,k_2,k_3}^{\iota_1,\iota_2,\iota_3}(\xi,s)\, d s\bigg|
	\lesssim \varepsilon_1^32^{-2p_0m}2^{-10k_+},
	\end{align}
	for \((\iota_1,\iota_2,\iota_3)\in \{(+,+,+),(+,-,-),(-,+,-),(-,-,+)\}\).
\end{lemma}

\begin{proof} We only present the proof of the case \((\iota_1,\iota_2,\iota_3)=(+,-,-)\), since the other cases can be handled analogously.
	Recalling \(k_1,k_2,k_3\in[k-20,k+20]\) and using \eqref{key inequality-1} and \eqref{key inequality-2}, we may estimate
	\begin{equation}\label{l42}
	\begin{aligned}
	|\Phi(\xi,\eta,\sigma)|\gtrsim 2^{(\alpha+1)k}.
	\end{aligned}
	\end{equation}
	Integration by parts in \(s\) as \eqref{l36}-\eqref{l37}, we are left to estimate \(A_j(\xi,t_j),\ j=1,2\) and \(B_j(\xi,s),\ j=1,2,3,4\). It follows from \eqref{l42} that
	\begin{equation*}
	\begin{aligned}
	\|\mathcal{F}^{-1}(m_2)\|_{L^1}\lesssim 2^{-(\alpha+1)k},
	\end{aligned}
	\end{equation*}
	and
	\begin{equation*}
	\begin{aligned}
	\|\mathcal{F}^{-1}(m_3)\|_{L^1}\lesssim \epsilon_1^22^{(1-2\alpha)k}2^{-m}2^{2k_{\pm}}.
	\end{aligned}
	\end{equation*}
	These two bounds together with  \eqref{ap2} allow us to estimate 
	\begin{equation*}
	\begin{aligned}
	A_j(\xi,t_j)
	\lesssim \varepsilon_1^32^{-\alpha k}2^{(2p_0-1/2)m}2^{-2N_0k_+}
	\lesssim \epsilon_1^32^{-m/4}2^{-10k_+},\quad j=1,2,
	\end{aligned}
	\end{equation*}
	\begin{equation*}
	\begin{aligned}
	\sup_{s\in[t_1,t_2]}B_4(\xi,s)
	&\lesssim \epsilon_1^52^{(2-2\alpha)k}2^{(2p_0-3/2)m}2^{-2N_0k_+}2^{2k_{\pm}}\\
	&\lesssim \epsilon_1^52^{-m}2^{-m/4}2^{-10k_+}.
	\end{aligned}
	\end{equation*}
	and 	
	\begin{equation*}
	\begin{aligned}
	\sup_{s\in[t_1,t_2]}\big[B_1(\xi,s)+B_2(\xi,s)+B_3(\xi,s)\big]
	&\lesssim \varepsilon_1^32^{-\alpha k}2^{(4p_0-3/2)m}2^{-N_0k_+}2^{2k_{\pm}}\\
	&\lesssim \epsilon_1^32^{-m}2^{-m/4}2^{-10k_+}.
	\end{aligned}
	\end{equation*}

\end{proof}

\begin{lemma}\label{le:2} It holds that
	\begin{equation}\label{31}
	\begin{aligned}
	&\quad\bigg|I_{k_1,k_2,k_3}^{+,+,-}(\xi,s)-\tilde{\tilde{c}}s^{-1}\xi|\xi|^{1-\alpha}\widehat{f_{k_1}}(\xi,s)\widehat{f_{k_2}}(\xi,s)\widehat{f_{k_3}}(-\xi,s)\bigg|\\
	&\quad+\bigg|I_{k_1,k_2,k_3}^{+,-,+}(\xi,s)-\tilde{\tilde{c}}s^{-1}\xi|\xi|^{1-\alpha}\widehat{f_{k_1}}(\xi,s)\widehat{f_{k_2}}(-\xi,s)\widehat{f_{k_3}}(\xi,s)\bigg|\\
	&\quad+\bigg|I_{k_1,k_2,k_3}^{-,+,+}(\xi,s)-\tilde{\tilde{c}}s^{-1}\xi|\xi|^{1-\alpha}\widehat{f_{k_1}}(-\xi,s)\widehat{f_{k_2}}(\xi,s)\widehat{f_{k_3}}(\xi,s)\bigg|\\
	&\lesssim \varepsilon_1^32^{-m}2^{-3p_0m}2^{-10k_+}.
	\end{aligned}
	\end{equation}
\end{lemma}

\begin{proof} The proof is identical to a similar one in \cite[Theorem 4.3]{SW1}, so we omit.   
	
\end{proof}

\section{Proof of Theorem \ref{th:main-2}}\label{Proof of mfNLS}

\subsection{Decay estimates}

We have the decay estimates for the solution of the equations \eqref{eq:main-2}:
\begin{lemma}\label{decay:fNLS}
	Let \(\alpha\in (-1,1)\setminus\{0\}\) and \(t\geq 1\). Assume that \(u\) is the solution of \eqref{eq:main-2} and satisfies
	\begin{align*}
	t^{-p_0}\|u\|_{H^{N_0}}+(1+t)^{-p_0}\|f\|_{H^{1,1}}+\|f\|_Z\leq 1,
	\end{align*}
	then it holds
	\begin{align}\label{d3}
	\|u\|_{L^\infty}+\|\partial_xu\|_{L^\infty}\lesssim t^{-1/2}.
	\end{align}
	
\end{lemma}

\begin{proof} Based on \eqref{d1} and \eqref{d2}, the proof is exactly same to the one in \cite[Lemma 2.2]{SW1} for \(\alpha\in (-1,0)\), and the one in Lemma \ref{decay:fKdV} for \(\alpha\in (0,1)\).

\end{proof}

From Lemma \ref{decay:fNLS}, we see that \eqref{fNLS-decay} is a consequence of \eqref{fNLS2}.

\subsection{The main estimates}

We first prove the uniform bounds for the energy parts in \eqref{l9}. 
\begin{proposition}\label{pr:1} Let \(u\) be a solution of the equation \eqref{eq:main-2} with initial data \eqref{eq:initial-1} satisfying the a priori bounds \eqref{l8}. Then the following estimates hold true:
	\begin{align}\label{6}
	\|u(t,\cdot)\|_{H^{N_0}}\leq C\varepsilon_0\langle t \rangle^{C\varepsilon_1^2},
	\end{align}
	and
	\begin{align}\label{7}
	\|f(t,\cdot)\|_{H^{1,1}}\leq C(\varepsilon_0+ \varepsilon_1^3)\langle t \rangle^{C\varepsilon_1^2}.
	\end{align}	
\end{proposition}

\begin{proof} For convenience, we define the short-hand notation \(\mathcal{N}(u)=-|u|^2u\).
	Using Duhamel's principle, one can express the solution as
	\begin{equation*}
	\begin{aligned}
	u(t)=e^{-\mathrm{i}t|D|^{\alpha+1}}u(1)+\mathrm{i}\int_1^te^{-\mathrm{i}(t-s)|D|^{\alpha+1}}\mathcal{N}(u)\, d s,
	\end{aligned}
	\end{equation*}
	which gives the following estimate
	\begin{equation*}
	\begin{aligned}
	\|u\|_{H^{N_0}}&\lesssim \|u(1)\|_{H^{N_0}}+\int_1^t\|u\|_{L^\infty}^2\|u\|_{H^{N_0}}\, d s\\
	&\lesssim \varepsilon_0+\int_1^t\varepsilon_1^2s^{-1}\|u\|_{H^{N_0}}\, d s.
	\end{aligned}
	\end{equation*}
The estimate \eqref{6} immediately follows.
	
The estimate \eqref{7} will be built upon some estimates involving the following operators  
	\begin{equation*}
	\begin{aligned}
	\mathcal{L}=\mathrm{i}\partial_t-|D|^{\alpha+1},\quad \mathcal{J}=e^{-\mathrm{i}t|D|^{\alpha+1}}xe^{\mathrm{i}t|D|^{\alpha+1}},
	\quad S=(\alpha+1)t\partial_t+x\partial_x.
	\end{aligned}
	\end{equation*}   
	It is straightforward to check that 
	\begin{align}
	[\mathcal{L},\mathcal{J}]=0,\quad [\mathcal{L},S]=\mathrm{i}(\alpha+1)\mathcal{L}.
	\end{align}
	Hence we may estimate 
	\begin{equation*}
	\begin{aligned}
	\|\mathcal{J}u\|_{L^2}
	&\lesssim\|\mathcal{J}u(1)\|_{L^2}+\int_1^t\|\mathcal{J}\mathcal{N}(u)\|_{L^2}\, d s\\
	&\lesssim \varepsilon_0+\int_1^t\varepsilon_1^2s^{-1}\|\mathcal{J}u\|_{L^2}\, d s,
	\end{aligned}
	\end{equation*}
	and
	\begin{equation*}
	\begin{aligned}
	\|Su\|_{L^2}
	&\lesssim \|Su(1)\|_{L^2}+\int_1^t\|[S+\mathrm{i}(\alpha+1)]\mathcal{N}(u)\|_{L^2}\, d s\\
	&\lesssim \varepsilon_0+\int_1^t\varepsilon_1^2s^{-1}\|Su\|_{L^2}\, d s,
	\end{aligned}
	\end{equation*}
	which yields 
	\begin{equation*}
	\begin{aligned}
	\|\mathcal{J}u\|_{L^2}+\|Su\|_{L^2}
	\leq C\varepsilon_0\langle t\rangle^{C\varepsilon_1^2}.
	\end{aligned}
	\end{equation*}
	
	The desired bound on \(\|xf\|_{L^2}\) is immediate due to \(\|xf\|_{L^2}=\|\mathcal{J}u\|_{L^2}\). It remains to control \(\|x\partial_xf\|_{L^2}\).
	To this end, 
	one first calculates
	\begin{equation*}
	\begin{aligned}
	\xi\partial_\xi \widehat{f}(\xi)
	=-e^{\mathrm{i}t|\xi|^{\alpha+1}}\big[\mathrm{i}(\alpha+1)t\widehat{\mathcal{N}(u)}+\widehat{Su}+\widehat{u}\big](\xi),
	\end{aligned}
	\end{equation*}
and then obtains
	\begin{equation*}
	\begin{aligned}
	\|\mathcal{F}(x\partial_xf)\|_{L^2}\leq \|\xi\partial\widehat{f}\|_{L^2}+\|\widehat{f}\|_{L^2}
	\leq C(\varepsilon_0+\varepsilon_1^3)\langle t\rangle^{C\varepsilon_1^2}.
	\end{aligned}
	\end{equation*}

\end{proof}

We now turn to proving the uniform bound for Z-norm part in \eqref{l9} which constitutes the main body of the remaining proof. 
Taking Fourier transform on \eqref{eq:main-2} gives
\begin{equation}\label{8}
\begin{aligned}
\partial_t\widehat{f}(\xi,t)=\frac{\mathrm{i}}{2\pi}\int_{\mathbb{R}^2}e^{\mathrm{i}t\Phi(\xi,\eta,\sigma)}
\widehat{f}(\xi-\eta,t)\widehat{f}(\eta-\sigma,t)\widehat{\overline{f}}(\sigma,t)\, d\eta d \sigma=J(\xi,t),
\end{aligned}
\end{equation}
in which  
\begin{equation*}\label{11}
\begin{aligned}
\Phi(\xi,\eta,\sigma)=|\xi|^{\alpha+1}-|\xi-\eta|^{\alpha+1}-|\eta-\sigma|^{\alpha+1}+|\sigma|^{\alpha+1}.
\end{aligned}
\end{equation*}
Let
\begin{align}\label{def:L}
L(\xi,t):=\frac{-|\xi|^{1-\alpha}}{|\alpha|(\alpha+1)}\int_1^t|\widehat{f}(\xi,s)|^2\,\frac{d s}{s},
\end{align}	
and
\begin{align*}
g(\xi,t):=e^{\mathrm{i}L(\xi,t)}\widehat{f}(\xi,t).
\end{align*}	
One substitutes these formulae into \eqref{8} to find that 
\begin{equation*}
\begin{aligned}
\partial_t g(\xi,t)&=e^{\mathrm{i}L(\xi,t)}\big[ J(\xi,t)-\mathrm{i}\tilde{c}t^{-1}|\xi|^{1-\alpha}|\widehat{f}(\xi,t)|^2\widehat{f}(\xi,t)\big],
\end{aligned}
\end{equation*}
where \(\tilde{c}:=[|\alpha|(\alpha+1)]^{-1}\). Then the \(Z\)-norm bound in \eqref{l9} is an immediate consequence of the following:
\begin{proposition}\label{pr:2}
	It holds that
	\begin{align}\label{9}
	t_1^{p_0}\left\|(|\xi|^{(1-\alpha)/4}+|\xi|^{10})\big[g(\xi,t_2)-g(\xi,t_1)\big]\right\|_{L^\infty_\xi}\lesssim\epsilon_0,
	\end{align} 
	for any \(t_1\leq t_2\in[1,T]\). 
\end{proposition}
  The inequality \eqref{9} also implies the estimate \eqref{fNLS3}.

\subsubsection{Reduction of \eqref{9}}
We let \(f_k^+=P_kf,f_k^-=P_k\bar{f}\), and decompose
\begin{equation}\label{9.5}
\begin{aligned}
J(\xi,t)=\mathrm{i}(2\pi)^{-1}\sum_{k_1,k_,2,k_3\in\mathbb{Z}}J_{k_1,k_,2,k_3}(\xi,t),
\end{aligned}
\end{equation}
where
\begin{equation}\label{10}
\begin{aligned}
J_{k_1,k_,2,k_3}(\xi,t)&=\int_{\mathbb{R}^2}e^{\mathrm{i}t\Phi(\xi,\eta,\sigma)}
\widehat{f_{k_1}^+}(\xi-\eta,t)\widehat{f_{k_2}^+}(\eta-\sigma,t)\widehat{f_{k_3}^-}(\sigma,t)\, d \eta d \sigma.
\end{aligned}
\end{equation}

For the proof of \eqref{9}, it suffices to show that for \(m\in \{1,2,\dots\}\) it holds
\begin{align}\label{12}
\left\|(|\xi|^{(1-\alpha)/4}+|\xi|^{10})\big[g(\xi,t_2)-g(\xi,t_1)\big]\right\|_{L^\infty_\xi}\lesssim\epsilon_02^{-p_0m},
\end{align}
for any \(t_1\leq t_2\in [2^m-1,2^{m+1}]\cap[1,T]\).
Arguing as Subsection \ref{reduction:fKdV}, one can first reduce the proof of \eqref{9} on the frequency \(|\xi|\in[2^k,2^{k+1}]\) for all \(k\in\mathbb{Z}\) to \(k\in [-10p_0m/(1-\alpha),p_0m]\cap \mathbb{Z}\).

Recalling the notation in \eqref{notation:positive-negative}, 
from the a priori assumptions \eqref{l8} and the localization, it follows that
\begin{equation}\label{13}
\begin{aligned}
&\|\widehat{f_l}(s)\|_{L^2}\lesssim \epsilon_12^{p_0m}\min(2^{-N_0l_+},2^{l/2}),\\
&\|\partial\widehat{f_l}(s)\|_{L^2}\lesssim \epsilon_12^{p_0m}\min\big[2^{-l},\max(2^{-l/2},1)\big],\\
&\|\widehat{f_l}(s)\|_{L^\infty}\lesssim \epsilon_12^{l_{\pm}},\\
&\|e^{-\mathrm{i}s|D|^{\alpha+1}}f_l(s)\|_{L^\infty}\lesssim \epsilon_12^{-m/2},
\end{aligned}
\end{equation}
for any  \(s\in[2^m-1,2^{m+1}]\cap[1,T]\) and any \(l\in\mathbb{Z}\).
Keeping \(k\in [-10p_0m/(1-\alpha),p_0m]\cap \mathbb{Z}\) in mind,  
and repeating the argument of Subsection \ref{reduction:fKdV} via \eqref{13}, one may further reduce \eqref{9} to the following proposition:

\begin{proposition}\label{pr:reduce}
	Assume that \(k,k_1,k_2,k_3\in\mathbb{Z}\), \(m\in\mathbb{Z}\cap[100,\infty)\), \(|\xi|\in[2^k,2^{k+1}]\), and \(t_1\leq t_2\in [2^m-1,2^{m+1}]\cap [1,T]\). If
	\begin{equation}\label{14}
	\begin{aligned}
	&k\in [-m/100,p_0m],\\
	&k_1,k_2,k_3\in[-4m,p_0m/10],\\
	&\min(k_1,k_2,k_3)+\mathrm{med}(k_1,k_2,k_3)\geq -(1+10p_0)m,
	\end{aligned} 
	\end{equation}
	then 
	\begin{equation}\label{15}
	\begin{aligned}
	&\bigg|\int_{t_1}^{t_2}e^{\mathrm{i}L(\xi,s)}\big[ J_{k_1,k_2,k_3}(\xi,s)-\tilde{\tilde{c}}s^{-1}|\xi|^{1-\alpha}\widehat{f_{k_1}}(\xi,s)\widehat{f_{k_2}}(\xi,s)\widehat{f_{k_3}}(-\xi,s)\big]\, d s\bigg|\\
	&\quad\quad\quad\quad\quad\quad\quad\quad\quad\quad\quad\quad\quad\quad\quad\lesssim
	\epsilon_1^32^{-p_0m}2^{-10k_+},
	\end{aligned}
	\end{equation}
	where \(\tilde{\tilde{c}}:=2\pi/[|\alpha|(\alpha+1)]\).
\end{proposition}

\subsubsection{Proof of Proposition \ref{pr:reduce}}

To prove Proposition \ref{pr:reduce}, we also split the frequencies into the non-resonant regime and resonant regime and use different techniques to attack it respectively.  Performing the change of variables \(\eta\mapsto-\eta\) and \(\sigma\mapsto-(\xi+\eta+\sigma)\), we may rewrite 
\begin{equation*}
\begin{aligned}
J_{k_1,k_,2,k_3}(\xi,s)&=\int_{\mathbb{R}^2}e^{\mathrm{i}s\tilde{\Phi}(\xi,\eta,\sigma)}\widehat{f_{k_1}^+}(\xi+\eta,s)
\widehat{f_{k_2}^+}(\xi+\sigma,s)\\
&\quad\times\widehat{f_{k_3}^-}
(-\xi-\eta-\sigma,s)\, d \eta d \sigma,
\end{aligned}
\end{equation*} 
where the new phase \(\tilde{\Phi}\) is given by
\begin{equation*}
\begin{aligned}
\tilde{\Phi}(\xi,\eta,\sigma)=|\xi|^{\alpha+1}-|\xi+\eta|^{\alpha+1}-|\xi+\sigma|^{\alpha+1}+|\xi+\eta+\sigma|^{\alpha+1}.
\end{aligned}
\end{equation*}  

\bigskip
\noindent{\bf{\emph{Non-resonant regime}: \(\max(|k_1-k|,|k_2-k|,|k_3-k|)\geq 21\).}} 

\begin{lemma}  If
	\begin{equation}\label{17}
	\begin{aligned}
	\min(k_1,k_2,k_3)\geq -(1-20p_0)m,
	\end{aligned}
	\end{equation}
	then the bound \eqref{15} holds.
\end{lemma}

\begin{proof}  It suffices to show
	\begin{equation*}
	\begin{aligned}
	|J_{k_1,k_,2,k_3}(\xi,s)|
	\lesssim \epsilon_1^32^{-m}2^{-3p_0m}2^{-10k_+}.
	\end{aligned}
	\end{equation*}

	\noindent{\bf{Case 1: \(\max(|k_1-k_2|,|k_1-k_3|,|k_2-k_3|)\geq 5\).}}
	By symmetry, we may assume that \(|k_1-k_3|\geq5\) and \(\max(k_1,k_3)\geq k-20\). We first consider the case of \(\alpha\in(0,1)\). 	
	Integrating by parts in \(\eta\) gives
	\begin{align}\label{18}
	|J_{k_1,k_,2,k_3}(\xi,s)|\leq |G_1(\xi,s)|+|G_2(\xi,s)|+|G_3(\xi,s)|,
	\end{align}
	where 
	\begin{equation*}
	\begin{aligned}
	G_1(\xi,s)&=\int_{\mathbb{R}^2}m_4(\eta,\sigma)e^{\mathrm{i}s\tilde{\Phi}(\xi,\eta,\sigma)}\partial_\eta\widehat{f_{k_1}^+}(\xi+\eta,s)
	\widehat{f_{k_2}^+}(\xi+\sigma,s)\\
	&\qquad\qquad\qquad\qquad\qquad\times\widehat{f_{k_3}^-}
	(-\xi-\eta-\sigma,s)\, d \eta d \sigma,\\
	G_2(\xi,s)&=\int_{\mathbb{R}^2} m_4(\eta,\sigma)e^{\mathrm{i}s\tilde{\Phi}(\xi,\eta,\sigma)}\widehat{f_{k_1}^+}(\xi+\eta,s)
	\widehat{f_{k_2}^+}(\xi+\sigma,s)\\
	&\qquad\qquad\qquad\qquad\qquad\times\partial_\eta\widehat{f_{k_3}^-}
	(-\xi-\eta-\sigma,s)\, d \eta d \sigma,\\
	G_3(\xi,s)&=\int_{\mathbb{R}^2}\partial_\eta m_4(\eta,\sigma)e^{\mathrm{i}s\tilde{\Phi}(\xi,\eta,\sigma)}\widehat{f_{k_1}^+}(\xi+\eta,s)
	\widehat{f_{k_2}^+}(\xi+\sigma,s)\\
	&\qquad\qquad\qquad\qquad\qquad\times\widehat{f_{k_3}^-}
	(-\xi-\eta-\sigma,s)\, d \eta d \sigma,\\
	\end{aligned}
	\end{equation*}
	in which 
	\begin{align}\label{19}
	m_4(\eta,\sigma):=\frac{1}{s\partial_\eta\tilde{\Phi}(\xi,\eta,\sigma)}\varphi_{k_1}^\prime(\xi+\eta)\varphi_{k_2}^\prime(\xi+\eta)\varphi_{k_3}^\prime(\xi+\eta+\sigma).
	\end{align}
	
	One first observes that
	\begin{equation}\label{20}
	\begin{aligned}
	|\partial_\eta\tilde{\Phi}(\xi,\eta,\sigma)|
	&=(\alpha+1)\big|
	|\xi+\eta|^{\alpha-1}(\xi+\eta)-|\xi+\eta+\sigma|^{\alpha-1}(\xi+\eta+\sigma)\big|\\
	&\gtrsim 2^{\alpha\max(k_1,k_3)}\gtrsim 2^{\alpha k},
	\end{aligned}
	\end{equation}
	on the support of the integral.  
	From \eqref{19} and \eqref{20}, it follows that 
	\begin{equation}\label{21}
	\begin{aligned}
	\|\mathcal{F}^{-1}(m_4)\|_{L^1}\lesssim 2^{-m}2^{-\alpha k},
	\end{aligned}
	\end{equation}
	and
	\begin{equation}\label{22}
	\begin{aligned}
	\|\mathcal{F}^{-1}(\partial_\eta m_4)\|_{L^1}\lesssim 2^{-m}2^{-\alpha k}2^{-\min(k_1,k_3)}.
	\end{aligned}
	\end{equation}
	To estimate \(G_1(\xi,s)\), we define
	\begin{equation*}
	\begin{aligned}
	&\widehat{g_1}(\eta):=e^{-\mathrm{i}s|\xi+\eta|^{(\alpha+1)}}\partial_\eta\widehat{f_{k_1}^+}(\xi+\eta,s),\\
	&\widehat{g_2}(\sigma):=e^{-\mathrm{i}s|\xi+\sigma|^{(\alpha+1)}}\widehat{f_{k_2}^+}(\xi+\sigma,s),\\
	&\widehat{g_3}(\theta):=e^{\mathrm{i}s|\xi-\sigma|^{(\alpha+1)}}\widehat{f_{k_3}^-}(-\xi+\theta,s),
	\end{aligned}
	\end{equation*} 	
and then have via \eqref{13}
	\begin{equation*}
	\begin{aligned}
	\|g_1\|_{L^2}\lesssim \epsilon_12^{p_0m}\max(2^{-k_1/2},1)
	,\ \|g_2\|_{L^2}
	\lesssim \epsilon_12^{p_0m}2^{-N_0{k_2}_+},\
	\|g_3\|_{L^\infty}\lesssim \epsilon_12^{-m/2},
	\end{aligned}
	\end{equation*}
This combines with \eqref{21} to give
	\begin{align*}
	|G_1(\xi,s)|
	\lesssim\epsilon_1^32^{-\alpha k}\max(2^{-k_1/2},1)2^{(2p_0-3/2)m}2^{-N_0\max({k_2}_+,{k_3}_+)}.
	\end{align*}
One similarly has
	\begin{align*}
	|G_2(\xi,s)|\lesssim\epsilon_1^32^{-\alpha k}\max(2^{-k_2/2},1)2^{(2p_0-3/2)m}2^{-N_0\max({k_1}_+,{k_3}_+)}.
	\end{align*}
	For \(G_3(\xi,s)\), we let
	\begin{equation*}
	\begin{aligned}
	&\widehat{g_1}(\eta):=e^{-\mathrm{i}s|\xi+\eta|^{(\alpha+1)}}\widehat{f_{k_1}
		^+}(\xi+\eta,s),\\
	&\widehat{g_2}(\sigma):=e^{-\mathrm{i}s|\xi+\sigma|^{(\alpha+1)}}\widehat{f_{k_2}^+}(\xi+\sigma,s),\\
	&\widehat{g_3}(\theta):=e^{\mathrm{i}s|\xi-\sigma|^{(\alpha+1)}}\widehat{f_{k_3}^-}(-\xi+\theta,s),
	\end{aligned}
	\end{equation*}
 and then estimate by \eqref{13}
	\begin{equation*}
	\begin{aligned}
	\|g_1\|_{L^2}\lesssim \epsilon_12^{p_0m}2^{k_1/2}
	,\ \|g_2\|_{L^2}
	\lesssim \epsilon_12^{p_0m}2^{-N_0{k_2}_+},\
	\|g_3\|_{L^\infty}\lesssim \epsilon_12^{-m/2}.
	\end{aligned}
	\end{equation*}
This together with \eqref{22} leads to 
	\begin{align*}
	|G_3(\xi,s)|
	\lesssim \epsilon_1^32^{-\alpha k}2^{-\min(k_1,k_3)/2}2^{(2p_0-3/2)m}2^{-N_0{k_2}_+}.
	\end{align*}
We conclude that
	\begin{equation*}
	\begin{aligned}
	&\quad|G_1(\xi,s)|+|G_2(\xi,s)|+|G_3(\xi,s)|\\
	&\lesssim \epsilon_1^32^{-\alpha k}\max(2^{-\min(k_1,k_2)/2},1)2^{(2p_0-3/2)m}2^{-10k_+}
	(2^{10\max({k_1}_+,{k_2}_+,{k_3}_+)}+1)\\
	&\lesssim\epsilon_1^32^{-m}2^{-5p_0m}2^{-10k_+},
	\end{aligned}
	\end{equation*}
where in the last line we used \eqref{14} and \eqref{17}.
	
	We next consider the case of \(\alpha\in(-1,0)\).  It is easy to see  that 
	\begin{equation*}
	\begin{aligned}
	|\partial_\eta\tilde{\Phi}(\xi,\eta,\sigma)|
	\gtrsim 2^{\alpha\min(k_1,k_3)},
	\end{aligned}
	\end{equation*}
on the support of the integral. 
Therefore we have
	\begin{equation*}
	\begin{aligned}
	\|\mathcal{F}^{-1}(m_4)\|_{L^1}\lesssim 2^{-m}2^{-\alpha\min(k_1,k_3)},
	\end{aligned}
	\end{equation*}
	and
	\begin{equation*}
	\begin{aligned}
	\|\mathcal{F}^{-1}(\partial_\eta m_4)\|_{L^1}\lesssim 2^{-m}2^{-\alpha \min(k_1,k_3)}2^{-\min(k_1,k_3)}.
	\end{aligned}
	\end{equation*}
	The term \(2^{-\alpha\min(k_1,k_3)}\) contributes an acceptable bound \(2^{p_0m/10}\) due to \(\alpha\in(-1,0)\) and \eqref{14}. Then the desired result follows repeating the argument of \(\alpha\in(0,1)\).  	

\bigskip
   \noindent{\bf{Case 2: \(\max(|k_1-k_2|,|k_1-k_3|,|k_2-k_3|)\leq 4\).}}
	In this case  \(\partial_\eta\tilde{\Phi}\neq0\) or \(\partial_\sigma\tilde{\Phi}\neq0\), and we will assume \(\partial_\eta\tilde{\Phi}\neq0\) without loss of generality. 
	Recalling \(2^{k_1}\approx 2^{k_2}\approx 2^{k_3}\) and using the properties of the support of the integral, one obtains
	\begin{equation*}
	\begin{aligned}
	|\partial_\eta\tilde{\Phi}(\xi,\eta,\sigma)|
	&=(\alpha+1)\big||\xi+\eta|^{\alpha-1}(\xi+\eta)-|\xi+\eta+\sigma|^{\alpha-1}(\xi+\eta+\sigma)\big|\\
	&\gtrsim 2^{\alpha k_2},
	\end{aligned}
	\end{equation*}
which gives rise to
	\begin{equation}\label{24}
	\begin{aligned}
	\|\mathcal{F}^{-1}(m_4)\|_{L^1}\lesssim 2^{-m}2^{-\alpha k_2},
	\end{aligned}
	\end{equation}
	and
	\begin{equation}\label{25}
	\begin{aligned}
	\|\mathcal{F}^{-1}(\partial_\eta m_4)\|_{L^1}\lesssim 2^{-m}2^{-(\alpha+1)k_2}.
	\end{aligned}
	\end{equation}
Using the bounds \eqref{24}	and \eqref{25}, and performing a similar manipulation as {\bf{Case 1}}, 
we may get
	\begin{equation*}
	\begin{aligned}
	|F_1(\xi,s)|+|F_2(\xi,s)|+|F_3(\xi,s)|
	&\lesssim\epsilon_1^32^{-(\alpha+1)k_2}2^{(2p_0-3/2)m}2^{-N_0{k_2}_+}\\
	&\lesssim \epsilon_1^32^{-m}2^{-5p_0m}2^{-10k_+},
	\end{aligned}
	\end{equation*}
in which we used \eqref{14} and \(-(\alpha+1)k_2\leq -(\alpha+1)k+14(\alpha+1)\).

\end{proof}

\begin{lemma} If
	\begin{equation}\label{26}
	\begin{aligned}
	\min(k_1,k_2,k_3)\leq -(1-20p_0)m,
	\end{aligned}
	\end{equation}
	then the bound \eqref{15} holds.
\end{lemma}
\begin{proof} It is enough to show
	\begin{align*}
	\bigg|\int_{t_1}^{t_2}e^{\mathrm{i}L(\xi,s)}J_{k_1,k_2,k_3}(\xi,s)\, d s\bigg|
	\lesssim \varepsilon_1^32^{-2p_0m}2^{-10k_+}.
	\end{align*}
	
Recalling 
	\begin{equation*}
	\begin{aligned}
\tilde{\Phi}(\xi,\eta,\sigma)
=|\xi|^{\alpha+1}-|\xi+\eta|^{\alpha+1}-|\xi+\sigma|^{\alpha+1}+|\xi+\eta+\sigma|^{\alpha+1},
	\end{aligned}
	\end{equation*}
we have the following two sub-cases to consider:\\
	{\bf{(i) \(k_3=\min(k_1,k_2,k_3)\).}} In this case,  one has
	\[
	k_3\in [-(1+20p_0)m,-(1-20p_0)m],\quad k_1,k_2\in [-30p_0m,p_0m/10].
	\]
Hence we have 
	\begin{equation}\label{lower bound-4}
	\begin{aligned}
	-\tilde{\Phi}(\xi,\eta,\sigma)
	&=\big[-|\xi|^{\alpha+1}+|\xi+\eta|^{\alpha+1}
	+|\eta|^{\alpha+1}\big]\\
	&\quad+\big[|\xi+\sigma|^{\alpha+1}-|\eta|^{\alpha+1}\big]
	-|\xi+\eta+\sigma|^{\alpha+1}\\
	&\gtrsim 2^{(\alpha+1)\min(k_1,k_2)}\geq 2^{-30(\alpha+1)p_0m},
	\end{aligned}
	\end{equation}
	on the support of the integral.
	
\noindent {\bf{(ii) \(k_3\neq \min(k_1,k_2,k_3)\).}} 
	In this case, we have
	\[
	k_1\in [-(1+20p_0)m,-(1-20p_0)m],\quad k_2,k_3\in [-30p_0m,p_0m/10].
	\]	
	Observing \(k\gg k_1\), one may estimate 
\begin{equation}\label{lower bound-5}
\begin{aligned}
\quad\tilde{\Phi}(\xi,\eta,\sigma)
&\geq \big[|\xi|^{\alpha+1}+|\xi+\eta+\sigma|^{\alpha+1}-|2\xi+\eta+\sigma|^{\alpha+1}\big]\\
&\quad-|\xi+\eta|^{\alpha+1}
-\big||\xi+\sigma|^{\alpha+1}-|2\xi+\eta+\sigma|^{\alpha+1}\big|\\
&\geq 2^{(\alpha+1)\min(k,k_3)-1}-2^{(\alpha+1)k_1+1}-2^{k_1+\alpha k_2+10}\\
&\gtrsim 2^{(\alpha+1)\min(k,k_3)-2}\geq \min(2^{-30(\alpha+1)p_0m},2^{-(\alpha+1)m/100}),
\end{aligned}
\end{equation}
on the support of the integral.

	We integrate by parts in \(s\) to estimate
	\begin{equation}\label{27}
	\begin{aligned}
	&\quad\bigg|\xi\int_{t_1}^{t_2}\int_{\mathbb{R}^2}e^{\mathrm{i}L(\xi,s)}
	e^{\mathrm{i}s\tilde{\Phi}(\xi,\eta,\sigma)}\widehat{f_{k_1}^+}(\xi+\eta,s)
	\widehat{f_{k_2}^+}(\xi+\sigma,s)\\
	&\qquad\qquad\qquad\qquad\qquad\qquad\times\widehat{f_{k_3}^-}
	(-\xi-\eta-\sigma,s)
	\, d \eta d \sigma d s\bigg|\\
	&\lesssim  
	\sum_{j=1}^2\bigg|\xi\int_{\mathbb{R}^2}\frac{e^{\mathrm{i}L(\xi,t_j)}
		e^{\mathrm{i}t_j\tilde{\Phi}(\xi,\eta,\sigma)}}{\tilde{\Phi}(\xi,\eta,\sigma)}\widehat{f_{k_1}^+}(\xi+\eta,t_j)
	\widehat{f_{k_2}^+}(\xi+\sigma,t_j)\\
	&\qquad\qquad\qquad\qquad\qquad\qquad\times\widehat{f_{k_3}^-}
	(-\xi-\eta-\sigma,t_j)\, d \eta d \sigma\bigg|\\
	&\quad+\int_{t_1}^{t_2}\bigg|\xi\int_{\mathbb{R}^2}\frac{e^{\mathrm{i}s\tilde{\Phi}(\xi,\eta,\sigma)}}{\tilde{\Phi}^(\xi,\eta,\sigma)}\frac{d}{d s}\bigg[e^{\mathrm{i}L(\xi,s)}\widehat{f_{k_1}^+}(\xi+\eta,s)\widehat{f_{k_2}^+}(\xi+\sigma,s)\\
	&\qquad\qquad\qquad\qquad\qquad\qquad\times
	\widehat{f_{k_3}^-}
	(-\xi-\eta-\sigma,s)\bigg]\, d \eta d \sigma \bigg| d s\\
	&=\colon \sum_{j=1}^2\tilde{A}_j(\xi,t_j)+\tilde{B}(\xi).
	\end{aligned}
	\end{equation}
	We break up the differential in \(s\) in \(\tilde{B}(\xi)\) to find that		
	\begin{equation}\label{28}
	\begin{aligned}
	\tilde{B}(\xi)\lesssim 2^m\sup_{s\in[t_1,t_2]}\big[\tilde{B}_1(\xi,s)+\tilde{B}_2(\xi,s)+\tilde{B}_3(\xi,s)+\tilde{B}_4(\xi,s)\big],
	\end{aligned}
	\end{equation}
	where 
	\begin{equation}\label{29}
	\begin{aligned}
	\tilde{B}_1(\xi,s)&= \bigg|\xi\int_{\mathbb{R}^2}m_5(\eta,\sigma)e^{\mathrm{i}s\tilde{\Phi}(\xi,\eta,\sigma)}\partial_s\widehat{f_{k_1}^+}(\xi+\eta,s)
	\widehat{f_{k_2}^+}(\xi+\sigma,s)\\
	&\qquad\qquad\qquad\qquad\qquad\qquad\times\widehat{f_{k_3}^-}
	(-\xi-\eta-\sigma,s)\,
	d \eta d \sigma\bigg|,\\
	\tilde{B}_2(\xi,s)&= \bigg|\xi\int_{\mathbb{R}^2}m_5(\eta,\sigma)e^{\mathrm{i}s\tilde{\Phi}(\xi,\eta,\sigma)}\widehat{f_{k_1}^+}(\xi+\eta,s)
	\partial_s\widehat{f_{k_2}^+}(\xi+\sigma,s)\\
	&\qquad\qquad\qquad\qquad\qquad\qquad\times\widehat{f_{k_3}^-}
	(-\xi-\eta-\sigma,s)\, d \eta d \sigma\bigg|,\\
	\tilde{B}_3(\xi,s)&= \bigg|\xi\int_{\mathbb{R}^2}m_5(\eta,\sigma)e^{\mathrm{i}s\tilde{\Phi}(\xi,\eta,\sigma)}\widehat{f_{k_1}
		^+}(\xi+\eta,s)
	\widehat{f_{k_2}^+}(\xi+\sigma,s)\\
	&\qquad\qquad\qquad\qquad\qquad\qquad\times\partial_s\widehat{f_{k_3}^-}
	(-\xi-\eta-\sigma,s)\, d \eta d \sigma\bigg|,\\
	\tilde{B}_4(\xi,s)&=\bigg|\xi\int_{\mathbb{R}^2}m_6(\eta,\sigma)e^{\mathrm{i}s\tilde{\Phi}(\xi,\eta,\sigma)}\widehat{f_{k_1}^+}(\xi+\eta,s)
	\widehat{f_{k_2}^+}(\xi+\sigma,s)\\
	&\qquad\qquad\qquad\qquad\qquad\qquad\times\widehat{f_{k_3}^-}
	(-\xi-\eta-\sigma,s)\,
	d \eta d \sigma\bigg|,
	\end{aligned}
	\end{equation}
	in which 
	\begin{equation}\label{30}
	\begin{aligned}
	m_5(\eta,\sigma):=\frac{1}{\tilde{\Phi}(\xi,\eta,\sigma)}\varphi_{k_1}^\prime(\xi-\eta-\sigma)\varphi_{k_2}^\prime(\eta)\varphi_{k_3}^\prime(\sigma).
	\end{aligned}
	\end{equation}
	and
	\begin{equation}\label{30.1}
	\begin{aligned}
	m_6(\eta,\sigma):=\partial_s H(\xi,s)m_5(\eta,\sigma).
	\end{aligned}
	\end{equation}
	
		It thus follows from \eqref{lower bound-4}-\eqref{lower bound-5} that
		\begin{equation*}
		\begin{aligned}
		\|\mathcal{F}^{-1}(m_5)\|_{L^1}\lesssim \min(2^{30(\alpha+1)p_0m},2^{(\alpha+1)m/100}).
		\end{aligned}
		\end{equation*}
		We therefore, in light of \eqref{ap2}, may estimate
		\begin{equation*}
		\begin{aligned}
		\tilde{A}_j(\xi,t_j)
		&\lesssim \epsilon_1^32^k\min(2^{30(\alpha+1)p_0m},2^{(\alpha+1)m/100})2^{(2p_0-1/2)m}2^{-N_0({k_1}_+,{k_2}_+,{k_3}_+)}\\
		&\lesssim \epsilon_1^32^{-m/4}2^{-10k_+},\quad \text{for}\ j=1,2,
		\end{aligned}
		\end{equation*}

	Recalling the definition \eqref{def:L}, one estimates 	
	\begin{equation*}
	\begin{aligned}
	|\partial_s L(\xi,s)|
	\lesssim \epsilon_1^22^{(2-\alpha)k}2^{-m}2^{2k_{\pm}},
	\end{aligned}
	\end{equation*}
	and then obtains
	\begin{equation*}
	\begin{aligned}
	\|\mathcal{F}^{-1}(m_6)\|_{L^1}\lesssim \epsilon_1^22^{(2-\alpha)k}\min(2^{30(\alpha+1)p_0m},2^{(\alpha+1)m/100})2^{-m}2^{2k_{\pm}}.
	\end{aligned}
	\end{equation*}
	Applying \eqref{ap2}, we may bound
	\begin{equation*}
	\begin{aligned}
	\sup_{s\in[t_1,t_2]}\tilde{B}_4(\xi,s)
	&\lesssim \epsilon_1^52^{(3-\alpha)k}\min(2^{30(\alpha+1)p_0m},2^{(\alpha+1)m/100})2^{(2p_0-3/2)m}\\
	&\quad\times 2^{-N_0({k_1}_+,{k_2}_+,{k_3}_+)}
	2^{2k_{\pm}}\\
	&\lesssim \epsilon_1^52^{-m}2^{-m/4}2^{-10k_+}.
	\end{aligned}
	\end{equation*}

	One first notices from \eqref{8} and \eqref{13} that  
	\begin{align*}
	\|\partial_s\widehat{f_l}(s)\|_{L^2}\lesssim \epsilon_1^32^{-m}2^{3p_0m}2^{-10l_+}.
	\end{align*}
	Let
	\begin{equation*}
	\begin{aligned}
	&\widehat{g_1}(\theta):=e^{-\mathrm{i}s|\xi-\theta|^\alpha(\xi-\theta)}\partial_s\widehat{f_{k_1}^+}(\xi-\theta,s),\\
	&\widehat{g_2}(\eta):=e^{-\mathrm{i}s|\eta|^\alpha\eta}\widehat{f_{k_2}^+}(\eta,s),\\
	&\widehat{g_3}(\sigma):=e^{\mathrm{i}s|\sigma|^\alpha\sigma}\widehat{f_{k_3}^-}(\sigma,s),
	\end{aligned}
	\end{equation*}
	we then use \eqref{13} to get
	\begin{equation*}
	\begin{aligned}
	\|g_1\|_{L^2}\lesssim \epsilon_12^{3p_0m}2^{-10{k_1}_{+}}2^{-m},\ \|g_2\|_{L^2}
	\lesssim \epsilon_12^{-N_0k_+}2^{p_0m},\
	\|g_3\|_{L^\infty}\lesssim \epsilon_12^{-m/2}.
	\end{aligned}
	\end{equation*}
	It thus follows from \eqref{ap2} that	
	\begin{equation*}
	\begin{aligned}
	\sup_{s\in[t_1,t_2]}\tilde{B}_1(\xi,s)
	&\lesssim \epsilon_1^32^k\min(2^{30(\alpha+1)p_0m},2^{(\alpha+1)m/100})2^{(4p_0-3/2)m}\\
	&\quad\times 2^{-10{k_1}_{+}}2^{-N_0({k_2}_+,{k_3}_+)}\\
	&\lesssim \epsilon_1^32^{-m}2^{-m/4}2^{-10k_+},
	\end{aligned}
	\end{equation*}
	Similarly one has
	\begin{equation*}
	\begin{aligned}
	\sup_{s\in[t_1,t_2]}\big[\tilde{B}_2(\xi,s)+\tilde{B}_3(\xi,s)\big]
	\lesssim \epsilon_1^32^{-m}2^{-m/4}2^{-10k_+}.
	\end{aligned}
	\end{equation*}

\end{proof}

\bigskip
\noindent{\bf{\emph{Resonant regime}: \(k_1,k_2,k_3\in[k-20,k+20]\).}} 

\begin{proof}[Proof of \eqref{15}]  It suffices to show
	\begin{equation*}
	\begin{aligned}
	&\bigg|J_{k_1,k_2,k_3}(\xi,s)-\tilde{\tilde{c}}s^{-1}|\xi|^{1-\alpha}\widehat{f_{k_1}}(\xi,s)\widehat{f_{k_2}}(\xi,s)\widehat{f_{k_3}}(-\xi,s)\bigg|\\
	&\lesssim \varepsilon_1^32^{-m}2^{-3p_0m}2^{-10k_+}.
	\end{aligned}
	\end{equation*}
	
	Let \(\bar{l}\) be the smallest integer with the property that 
	\begin{align*}
	2^{\bar{l}}\geq 2^{(1-\alpha)k/2}2^{-49m/100}.
	\end{align*}
We may decompose 
	\begin{align*}
	J_{k_1,k_2,k_3}(\xi,s)=\sum_{l_1,l_2=\bar{l}}^{k+20}J_{l_1,l_2}(\xi,s),
	\end{align*}
	with
	\begin{equation*}
	\begin{aligned}
	J_{l_1,l_2}(\xi,s)&=\int_{\mathbb{R}^2}e^{\mathrm{i}s\tilde{\Phi}(\xi,\eta,\sigma)}\widehat{f_{k_1}^+}(\xi+\eta,s)\widehat{f_{k_2}^+}(\xi+\sigma,s)\\
	&\quad\times\widehat{f_{k_3}^-}
	(-\xi-\eta-\sigma,s)\varphi_{l_1}^{(\bar{l})}(\eta)\varphi_{l_2}^{(\bar{l})}(\sigma)\,d \eta d \sigma,
	\end{aligned}
	\end{equation*}
	for any \(l_1,l_2\geq \bar{l}\), and where 
	\begin{equation*}
	\begin{aligned}
	\varphi_{k}^{(l)}(x):=\varphi(x/2^k),\  \text{if}\  k=l,
	\end{aligned}
	\end{equation*}
	and
	\begin{equation*}
	\begin{aligned}
	\varphi_{k}^{(l)}(x):=\varphi(x/2^k)-\varphi(x/2^{k-1}),\  \text{if}\  k\geq l+1.
	\end{aligned}
	\end{equation*}
	
	\bigskip
	\noindent{\bf{Case 1: \(l_2\geq \max(l_1,\bar{l}+1)\) or \(l_1\geq \max(l_2,\bar{l}+1)\).}} We only consider the case \(l_2\geq \max(l_1,\bar{l}+1)\), a similar argument  applies to the other case. 
	In this case, we only need to show
	\begin{align*}
	|J_{l_1,l_2}(\xi,s)|\lesssim \epsilon_1^32^{-m}2^{-3p_0m}2^{-10k_{+}}.
	\end{align*}
	
	On the support of the integral, one sees that
	\begin{equation}\label{34}
	\begin{aligned}
	|\partial_\eta\tilde{\Phi}(\xi,\eta,\sigma)|
	&=(\alpha+1)\big||\xi+\eta+\sigma|^{\alpha-1}(\xi+\eta+\sigma)-
	|\xi+\eta|^{\alpha-1}(\xi+\eta)\big|\\
	&\gtrsim 2^{l_2}2^{(\alpha-1)k}.
	\end{aligned}
	\end{equation}
Then we use integration by parts in \(\eta\) to obtain
	\begin{align*}
	|J_{l_1,l_2}(\xi,s)|\leq |G_4(\xi,s)|+|G_5(\xi,s)|+|G_6(\xi,s)|,
	\end{align*}
	where 
	\begin{equation*}
	\begin{aligned}
	G_4(\xi,s)&=\int_{\mathbb{R}^2}m_7(\eta,\sigma)e^{\mathrm{i}s\tilde{\Phi}(\xi,\eta,\sigma)}\widehat{f_{k_1}^+}(\xi+\eta,s)\widehat{f_{k_2}^+}(\xi+\sigma,s)\\
	&\qquad\qquad\qquad\qquad\qquad\times\partial_\eta\widehat{f_{k_3}^-}
	(-\xi-\eta-\sigma,s) \, d \eta d \sigma,\\
	G_5(\xi,s)&=\int_{\mathbb{R}^2}m_7(\eta,\sigma)e^{\mathrm{i}s\tilde{\Phi}(\xi,\eta,\sigma)}\widehat{f_{k_1}^+}(\xi+\eta,s)\widehat{f_{k_2}^+}(\xi+\sigma,s)\\
	&\qquad\qquad\qquad\qquad\qquad\times\widehat{f_{k_3}^-}
	(-\xi-\eta-\sigma,t)\, d \eta d \sigma,\\
	G_6(\xi,s)&=\int_{\mathbb{R}^2}\partial_\eta m_7(\eta,\sigma)e^{\mathrm{i}s\tilde{\Phi}(\xi,\eta,\sigma)}\partial_\eta\widehat{f_{k_1}^+}(\xi+\eta,s)\widehat{f_{k_2}^+}(\xi+\sigma,s)\\
	&\qquad\qquad\qquad\qquad\qquad\times\widehat{f_{k_3}^-}
	(-\xi-\eta-\sigma,s)\, d \eta d \sigma,
	\end{aligned}
	\end{equation*}
	with
	\begin{align}\label{35}
	m_7(\eta,\sigma):=\frac{\varphi_{l_1}^{(\bar{l})}(\eta)\varphi_{l_2}^{(\bar{l})}(\sigma)}{s\partial_\eta\tilde{\Phi}(\xi,\eta,\sigma)}
	\varphi_{k_1}^\prime(\xi+\eta)\varphi_{k_2}^\prime(\xi+\sigma)
	\varphi_{k_3}^\prime(\xi+\eta+\sigma).
	\end{align}

	Following \eqref{34} and \eqref{35}, a straightforward calculation shows that 
	\begin{equation*}
	\begin{aligned}
	|\partial_\eta^a\partial_\sigma^bm_7(\eta,\sigma)|&\lesssim (2^{-m}2^{-l_2}2^{(1-\alpha)k})(2^{-al_1}2^{-bl_2})\\
	&\quad \times{\bf{1}}_{[0,2^{l_1+4}]}(|\xi-\eta|){\bf{1}}_{[2^{l_2-4},2^{l_2+4}]}(|\xi-\sigma|),
	\end{aligned}
	\end{equation*}
	for \(a,b\in[0,20]\cap\mathbb{Z}\). Hence one has
	\begin{equation*}
	\begin{aligned}
	\|\mathcal{F}^{-1}(m_7)\|_{L^1}\lesssim 2^{-m}2^{-l_2}2^{(1-\alpha)k}.
	\end{aligned}
	\end{equation*}
	
	We first estimate the term \(G_4(\xi,s)\). 
	Fix \(\xi\) and \(s\), and let 
	\begin{equation*}
	\begin{aligned}
	&\widehat{g_1}(\eta):=e^{-\mathrm{i}s|\xi+\eta|^{(\alpha+1)}}\partial_\eta\widehat{f_{k_1}^+}(\xi+\eta,s),\\
	&\widehat{g_2}(\sigma):=e^{-\mathrm{i}s|\xi+\sigma|^{(\alpha+1)}}\widehat{f_{k_2}^+}(\xi+\sigma,s)\varphi_{l_2}(\sigma/2^{l_2+4}),\\
	&\widehat{g_3}(\theta):=e^{\mathrm{i}s|\xi-\sigma|^{(\alpha+1)}}\widehat{f_{k_3}^-}(-\xi+\theta,s)\varphi_{l_2}(\theta/2^{l_2+4}).
	\end{aligned}
	\end{equation*}
we then use \eqref{13} to get
	\begin{equation*}
	\begin{aligned}
	\|g_1\|_{L^2}\lesssim \epsilon_12^{-k}2^{p_0m},\quad
	\|g_2\|_{L^\infty}\lesssim \epsilon_12^{-m/2},\quad
	\|g_3\|_{L^2}\lesssim \epsilon_12^{l_2/2}2^{k_{\pm}}.
	\end{aligned}
	\end{equation*}
	It then follows from \eqref{ap2} that
	\begin{equation*}
	\begin{aligned}
	|G_4(\xi,s)|&\lesssim \|\mathcal{F}^{-1}(m_7)\|_{L^1}\|f\|_{L^2}\|g\|_{L^\infty}\|h\|_{L^2}\\
	&\lesssim \epsilon_1^32^{-3m/2+p_0m+49m/200}2^{(1+\alpha)k/2}2^{k_{\pm}}\\
	&\lesssim \epsilon_1^32^{-m}2^{-m/200}2^{-10k_+},
	\end{aligned}
	\end{equation*}
	which is stronger than what we need. A similar argument yields 
	\begin{align*}
	|G_5(\xi,s)|\lesssim \epsilon_1^32^{-m}2^{-m/200}2^{-10k_+}.
	\end{align*}

	To estimate the term \(G_6(\xi,s)\), we integrate by parts in \(\eta\) again to deduce
	\begin{align*}
	|G_6(\xi,s)|\leq |G_7(\xi,s)|+|G_8(\xi,s)|+|G_9(\xi,s)|,
	\end{align*}
	in which 
	\begin{equation*}
	\begin{aligned}
	G_7(\xi,s)&=\int_{\mathbb{R}^2}m_8(\eta,\sigma)e^{\mathrm{i}s\tilde{\Phi}(\xi,\eta,\sigma)}\widehat{f_{k_1}^+}(\xi+\eta,s)\widehat{f_{k_2}^+}(\xi+\sigma,s)\\
	&\qquad\qquad\qquad\qquad\qquad\times\partial_\eta\widehat{f_{k_3}^-}
	(-\xi-\eta-\sigma,s)\, d \eta d \sigma,\\
	G_8(\xi,s)&=\int_{\mathbb{R}^2}m_8(\eta,\sigma)e^{\mathrm{i}s\tilde{\Phi}(\xi,\eta,\sigma)}\widehat{f_{k_1}^+}(\xi+\eta,s)\widehat{f_{k_2}^+}(\xi+\sigma,s)\\
	&\qquad\qquad\qquad\qquad\qquad\times\widehat{f_{k_3}^-}
	(-\xi-\eta-\sigma,s) \, d \eta d \sigma,\\
	G_9(\xi,s)&=\int_{\mathbb{R}^2}\partial_\eta m_8(\eta,\sigma)e^{\mathrm{i}s\tilde{\Phi}(\xi,\eta,\sigma)}\partial_\eta\widehat{f_{k_1}^+}(\xi+\eta,s)\widehat{f_{k_2}^+}(\xi+\sigma,s)\\
	&\qquad\qquad\qquad\qquad\qquad\times\widehat{f_{k_3}^-}
	(-\xi-\eta-\sigma,s)\, d \eta d \sigma,
	\end{aligned}
	\end{equation*}
	with
	\begin{align}\label{36}
	m_8(\eta,\sigma):=\frac{\partial_\eta m_7(\eta,\sigma)}{s\partial_\eta\tilde{\Phi}(\xi,\eta,\sigma)}.
	\end{align}
	
	It follows from \eqref{34} and \eqref{36} that \(m_8\) satisfies the following stronger estimate 
	\begin{equation*}
	\begin{aligned}
	|\partial_\eta^a\partial_\sigma^bm_8(\eta,\sigma)|&\lesssim (2^{-m}2^{-l_1-l_2}2^{(1-\alpha)k})(2^{-m}2^{-l_2}2^{(1-\alpha)k})(2^{-al_1}2^{-bl_2})\\
	&\quad\times{\bf{1}}_{[0,2^{l_1+4}]}(|\xi-\eta|){\bf{1}}_{[2^{l_2-4},2^{l_2+4}]}(|\xi-\sigma|),
	\end{aligned}
	\end{equation*}
	for \(a,b\in[0,19]\cap\mathbb{Z}\).
	In a similar fashion as \(G_4(\xi,s)\) and \(G_5(\xi,s)\), we use \eqref{ap2} to obtain
	\begin{equation*}
	\begin{aligned}
	|G_7(\xi,s)|+|G_8(\xi,s)|
	\lesssim \epsilon_1^32^{-m}2^{-m/200}2^{-10k_+}.
	\end{aligned}
	\end{equation*}
	We finally estimate the left term \(G_9(\xi,s)\) as follows:
	\begin{equation*}
	\begin{aligned}
	|G_9(\xi,s)|&\lesssim  2^{l_1}2^{l_2}|\partial_\eta m_8|\big\|\widehat{f_{k_1}^+}\big\|_{L^\infty}\big\|\widehat{f_{k_2}^+}\big\|_{L^\infty}\big\|\widehat{f_{k_3}^-}\big\|_{L^\infty}\\
	&\lesssim \epsilon^3(2^{-m}2^{-l_1-l_2}2^{(1-\alpha)k})2^{-m}2^{(1-\alpha)k}2^{3k_{\pm}}\\
	&\lesssim \epsilon^32^{-m}2^{-m/200}2^{-10k_+}.
	\end{aligned}
	\end{equation*}
	
	\bigskip
	\noindent{\bf{Case 2: \(l_1=l_2=\bar{l}\).}} In this case, it suffices to prove that
	\begin{align}\label{37}
	\big|J_{\bar{l},\bar{l}}(\xi,s)-\tilde{\tilde{c}}s^{-1}|\xi|^{1-\alpha}\widehat{f_{k_1}^+}(\xi,s)\widehat{f_{k_2}^+}(\xi,s)\widehat{f_{k_3}^-}(-\xi,s)\big|
	\lesssim \epsilon_1^32^{-m}2^{-2p_0m}2^{-10k_+}.
	\end{align} 
	Defining
	\begin{equation*}
	\begin{aligned}
	\tilde{J}_{\bar{l},\bar{l}}(\xi,s)&=\int_{\mathbb{R}^2}e^{\mathrm{i}s\alpha(\alpha+1)\eta\sigma/|\xi|^{1-\alpha}}\widehat{f_{k_1}^+}(\xi+\eta,s)\widehat{f_{k_2}^+}(\xi+\sigma,s)\\
	&\quad\times\widehat{f_{k_3}^-}
	(-\xi-\eta-\sigma,s)\varphi\big(2^{-\bar{l}}\eta\big)\varphi\big(2^{-\bar{l}}\sigma\big)\, d\eta d \sigma,
	\end{aligned}
	\end{equation*}
	and observing
	\begin{align*}
	\tilde{\Phi}(\xi,\eta,\sigma)=\alpha(\alpha+1)|\xi|^{\alpha-1}\eta\sigma
	+2^{(\alpha-2)k}\mathcal{O}\big[(|\eta|+|\sigma|)^3\big],
	\end{align*} 
	for \(|\xi-\eta|+|\xi-\sigma|\leq 2^{k-5}\), 
	we  estimate
	\begin{equation}\label{38}
	\begin{aligned}
	|J_{\bar{l},\bar{l}}(\xi,s)-\tilde{J}_{\bar{l},\bar{l}}(\xi,s)|
	\lesssim \epsilon^32^m2^{(\alpha-2)k}2^{5\bar{l}}2^{3k_{\pm}}
	\lesssim \epsilon^32^{-m}2^{-2m/5}2^{-10k_+}.
	\end{aligned}
	\end{equation}
	Noticing
	\begin{equation*}
	\begin{aligned}
	\big|\widehat{f_{k}}(\xi+r,s)-\widehat{f_{k}}(\xi,s)\big|\lesssim 2^{\bar{l}/2}2^{-k}2^{p_0m},\quad \text{for}\ |r|\leq 2^{\bar{l}},
	\end{aligned}
	\end{equation*}
	we obtain
	\begin{equation*}
	\begin{aligned}
	&\big|\widehat{f_{k_1}^+}(\xi+\eta,s)\widehat{f_{k_2}^+}(\xi+\sigma,s)\widehat{f_{k_3}^-}
	(-\xi-\eta-\sigma,s)-\widehat{f_{k_1}^+}(\xi,s)\widehat{f_{k_2}^+}(\xi,s)\widehat{f_{k_3}^-}(-\xi,s)\big|\\
	&\lesssim \epsilon^32^{\bar{l}/2}2^{p_0m}2^{-k}2^{3k_{\pm}},\quad \text{for}\ |\xi-\eta|+|\xi-\sigma|\leq 2^{\bar{l}+4}.
	\end{aligned}
	\end{equation*}
	It then follows that
	\begin{equation}\label{39}
	\begin{aligned}
	&\bigg|\tilde{J}_{\bar{l},\bar{l}}(\xi,s)-\int_{\mathbb{R}^2}e^{\mathrm{i}s\alpha(\alpha+1)\eta\sigma/|\xi|^{1-\alpha}}\widehat{f_{k_1}^+}(\xi,s)\widehat{f_{k_2}^+}(\xi,s)\widehat{f_{k_3}^-}(-\xi,s)\\
	&\qquad\qquad\qquad\qquad\qquad\qquad\qquad\times\varphi\big(2^{-\bar{l}}\eta\big)\varphi\big(2^{-\bar{l}}\sigma\big)\, d\eta d \sigma\bigg|\\
	&\lesssim \epsilon^32^{\bar{l}/2}2^{p_0m}2^{2\bar{l}}2^{3k_{\pm}}
	\lesssim \epsilon^32^{-6m/5}2^{-10k_+}.
	\end{aligned}
	\end{equation}

	One calculates  
	\begin{equation*}
	\begin{aligned}
	\int_{\mathbb{R}^2} e^{-\mathrm{i}xy}e^{-x^2/N^2}e^{-y^2/N^2}\, d x d y
	=\frac{2\pi N}{\sqrt{4N^{-2}+N^2}}=2\pi+\mathcal{O}(N^{-1}),
	\end{aligned}
	\end{equation*}
	in which we used the formula 
	\begin{equation*}
	\begin{aligned}
	\int_\mathbb{R} e^{-ax^2-bx}\, d x=e^{b^2/(4a)}\sqrt{\pi/a}, \quad \mathrm{for}\ a,b\in \mathcal{C}, \Re a>0.
	\end{aligned}
	\end{equation*}
	It  follows that
	\begin{equation*}
	\begin{aligned}
	\int_{\mathbb{R}^2} e^{-\mathrm{i}xy}\varphi(x/N)\varphi(y/N)\, d x d y
	=2\pi+\mathcal{O}(N^{-1/2}),\quad \text{for}\ N\geq 1.
	\end{aligned}
	\end{equation*}
	We then  estimate 
	\begin{equation*}
	\begin{aligned}
	&\left|\int_{\mathbb{R}^2}e^{\mathrm{i}s\alpha(\alpha+1)\eta\sigma/|\xi|^{1-\alpha}}\varphi\big(2^{-\bar{l}}\eta\big)\varphi\big(2^{-\bar{l}}\sigma\big)\, d \eta d \sigma-\frac{2\pi|\xi|^{1-\alpha}}{s|\alpha|(\alpha+1)}\right|\\
	&\lesssim 2^{-\alpha k}2^{-m}\big(2^{m/100}\big)^{-1/2}
	\lesssim 2^{-m}2^{-m/300}.
	\end{aligned}
	\end{equation*}
	Hence we get
	\begin{equation}\label{40}
	\begin{aligned}
	&\bigg|\int_{\mathbb{R}^2}e^{\mathrm{i}s\alpha(\alpha+1)\eta\sigma/|\xi|^{1-\alpha}}\widehat{f_{k_1}^+}(\xi,s)\widehat{f_{k_2}^+}(\xi,s)\widehat{f_{k_3}^-}(-\xi,s)\varphi\big(2^{-\bar{l}}\eta\big)\varphi\big(2^{-\bar{l}}\sigma\big)\, d \eta d \sigma\\
	&\qquad\qquad\qquad\qquad\qquad\qquad-\frac{2\pi |\xi|^{1-\alpha}}{s|\alpha|(\alpha+1)}\widehat{f_{k_1}^+}(\xi,s)\widehat{f_{k_2}^+}(\xi,s)\widehat{f_{k_3}^-}(-\xi,s)\bigg|\\
	&\lesssim \epsilon^32^{-m}2^{-m/300}2^{3k_{\pm}}
	\lesssim \epsilon^32^{-m}2^{-m/300}2^{-10k_+}.
	\end{aligned}
	\end{equation}
	
	We finally conclude \eqref{37} from  \eqref{38}-\eqref{40}.
\end{proof}

\section{Appendix}

 We collect here some technical results for which we refer  for instance to \cite[Lemma 2.1, Lemma 2.2, Lemma 2.3]{MR3961297}:
 \begin{lemma}
 	It holds that 
 	\begin{align}\label{ap2}
 	\left|\int_{\mathbb{R}^2}m(\eta,\sigma)\widehat{f}(\eta)\widehat{g}(\sigma)\widehat{h}(-\eta-\sigma)\, d \eta d \sigma\right|\lesssim \|\mathcal{F}^{-1}(m)\|_{L^1}\|f\|_{L^p}\|g\|_{L^q}\|h\|_{L^r},
 	\end{align}
 	for any \((p,q,r)\in\{(2,2,\infty),(2,\infty,2),(\infty,2,2)\}\).
 \end{lemma}
 
 \begin{lemma} It holds that 
 	\begin{align}\label{ap1}
 	\|\widehat{P_kg}\|_{L^\infty}\lesssim \left\|P_kg\right\|_{L^1}\lesssim 2^{-k/2}\|\widehat{P_kg}\|_{L^2}^{1/2}(\|\widehat{P_kg}\|_{L^2}+2^k\|\partial\widehat{P_kg}\|_{L^2})^{1/2}.
 	\end{align}
 \end{lemma}

 \begin{lemma} Let \(\alpha\in (-1,1)\setminus\{0\}\).
 	For any \(t\geq 1\),  the following linear dispersive estimates hold:
 	\begin{align}\label{l4}
 	\left\|e^{t|D|^{\alpha}\partial_x}P_kg\right\|_{L^\infty}
 	\lesssim t^{-\frac{1}{2}}2^{\frac{1-\alpha}{2}k}\|\widehat{g}\|_{L^\infty}
 	+t^{-\frac{3}{4}}2^{-\frac{1+3\alpha}{4}k}(\|\widehat{g}\|_{L^2}+2^k\|\partial\widehat{g}\|_{L^2}),
 	\end{align}
 	and
 	\begin{align}\label{l5}
 	\left\|e^{t|D|^{\alpha}\partial_x}P_kg\right\|_{L^\infty}
 	\lesssim t^{-\frac{1}{2}}2^{\frac{1-\alpha}{2}k}\|g\|_{L^1}.
 	\end{align}
 	
 \end{lemma}

To study the decay of solutions to \eqref{eq:main-2}, we need the following dispersive linear estimate on the fNLS semi-group:
\begin{lemma} Let \(\alpha\in (-1,1)\setminus\{0\}\).
	For any \(t\geq 1\),  the following linear dispersive estimates hold:
	\begin{align}\label{d1}
	\left\|e^{-\mathrm{i}t|D|^{\alpha+1}}P_kg\right\|_{L^\infty}
	\lesssim t^{-\frac{1}{2}}2^{\frac{1-\alpha}{2}k}\|\widehat{g}\|_{L^\infty}
	+t^{-\frac{3}{4}}2^{-\frac{1+3\alpha}{4}k}(\|\widehat{g}\|_{L^2}+2^k\|\partial\widehat{g}\|_{L^2}),
	\end{align}
	and
	\begin{align}\label{d2}
	\left\|e^{-\mathrm{i}t|D|^{\alpha+1}}P_kg\right\|_{L^\infty}
	\lesssim t^{-\frac{1}{2}}2^{\frac{1-\alpha}{2}k}\|g\|_{L^1}.
	\end{align}
	
\end{lemma}

\begin{proof} The proof is similar to that of \cite[Lemma 2.3]{MR3961297}, but we include it  for the sake of completeness. The estimate \eqref{d2} is an easy consequence of \eqref{d1}, so we only prove \eqref{d1}. We write
	\begin{align*}
	e^{-\mathrm{i}t|D|^{\alpha+1}}P_k\phi(x)=\frac{1}{\sqrt{2\pi}}\int_\mathbb{R} e^{\mathrm{i}(x\xi-t|\xi|^{\alpha+1})}\widehat{P_k\phi}(\xi)\, d \xi.
	\end{align*}
	Observing  
	\begin{align*}
	\left|\int_\mathbb{R} e^{\mathrm{i}(x\xi-t|\xi|^{\alpha+1})}\widehat{P_k\phi}(\xi)\, d \xi\right|\lesssim 2^k\|\widehat{P_k\phi}\|_{L^\infty},
	\end{align*}
	we have for \(t\lesssim 2^{-k\alpha}\) that
	\begin{align*}
	\|e^{-\mathrm{i}t|D|^{\alpha+1}}P_k\phi\|_{L^\infty}\lesssim |t|^{-1}2^{k(1/2-\alpha)}\|\widehat{P_k\phi}\|_{L^\infty}.
	\end{align*}
	Therefore we only need to consider \(t\gtrsim 2^{-k\alpha}\) in the following.
	
	Let
	\begin{align*}
	\mathcal{I}:=\{k\in \mathbb{Z}:\ (\alpha+1)^{-1}2^{\alpha-2} |tx^{-1}|\leq 2^{k\alpha}\leq (\alpha+1)^{-1}2^{2-\alpha}|tx^{-1}|\}.
	\end{align*} 
	
	\noindent{\bf{Case 1: \(k\in \mathbb{Z}\setminus\mathcal{I}\).}} 
	Observing 
	\begin{align*}
	|x-t(\alpha+1)|\xi|^{\alpha-1}\xi|\gtrsim t2^{-k(1-\alpha)},
	\end{align*}
	we use integration by parts to deduce
	\begin{equation*}
	\begin{aligned}
	&\left|\int_\mathbb{R} e^{\mathrm{i}(x\xi-t|\xi|^{\alpha+1})}\widehat{P_k\phi}(\xi)\, d \xi\right|\\
	&\lesssim \int_\mathbb{R} \frac{|\partial\widehat{P_k\phi}(\xi)|}{|x-t(\alpha+1)|\xi|^{\alpha-1}\xi|}\, d \xi+\int_\mathbb{R} \frac{|t\alpha(\alpha+1)|\xi|^{\alpha-1}||\widehat{P_k\phi}(\xi)|}{|x-t(\alpha+1)|\xi|^{\alpha-1}\xi|^{\alpha}|^2}\, d \xi\\
	&\lesssim t^{-1}2^{k(1/2-\alpha)}(2^k\|\partial\widehat{P_k\phi}\|_{L^2}+\|\widehat{P_k\phi}\|_{L^2})\\
	&\lesssim t^{-3/4}2^{-k(3\alpha/4-1/2)}(2^k\|\partial\widehat{P_k\phi}\|_{L^2}+\|\widehat{P_k\phi}\|_{L^2}),
	\end{aligned}
	\end{equation*}
	where we have used \(t\gtrsim 2^{-k\alpha}\) in the last inequality.\\	
	
	\noindent{\bf{Case 2: \(k\in \mathcal{I}\).}} It is easy to see that there is a unique \(\xi_0\in\mathbb{R}\) satisfying \(x-t(\alpha+1)|\xi|^{\alpha-1}\xi=0\) and \(|\xi_0|\approx 2^k\). Let \(l_0\) be the smallest integer satisfying  \(2^{l_0}\geq |t|^{-1/2}2^{k(1-\alpha/2)}\).  Then, one has
	\begin{align*}
	\left|\int_{\mathbb{R}} e^{\mathrm{i}(x\xi-t|\xi|^{\alpha+1})}\widehat{P_k\phi}(\xi)\, d \xi\right|
	&\leq \left|\int_{\mathbb{R}} e^{\mathrm{i}(x\xi-t|\xi|^{\alpha+1})}\widehat{P_k\phi}(\xi)\varphi_{l_0}\big(\xi-\xi_0\big)\, d \xi\right|\\
	&\quad+\sum_{l\geq l_0+1}\left|\int_{\mathbb{R}} e^{\mathrm{i}(x\xi-t|\xi|^{\alpha+1})}\widehat{P_k\phi}(\xi)\psi_l(\xi-\xi_0)\, d \xi\right|\\
	&=\colon J_{l_0}+\sum_{l\geq l_0+1}J_l,
	\end{align*}
	where \(l\geq l_0+1\).
	It is easy to see that 
	\begin{align*}
	J_{l_0}
	\leq t^{-1/2}2^{k(1-\alpha/2)}\|\widehat{P_k\phi}\|_{L^\infty}.
	\end{align*}
	It remains to bound \(J_l\) for \(l\geq l_0+1\). For this, notice that
	\begin{align*}
	|x-t(\alpha+1)|\xi|^{\alpha-1}\xi|\gtrsim t2^{l-k(2-\alpha)},
	\end{align*}
	we then integrate by parts to deduce
	\begin{align*}
	J_l
	&\leq \int_{\mathbb{R}} \frac{|\partial\widehat{P_k\phi}(\xi)\psi_l(\xi-\xi_0)|}{|x-t(\alpha+1)|\xi|^{\alpha-1}\xi|}\, d \xi
	+\int_{\mathbb{R}} \frac{|\widehat{P_k\phi}(\xi)\partial\psi_l(\xi-\xi_0)|}{|x-t(\alpha+1)|\xi|^{\alpha-1}\xi|}\, d \xi\\
	&\quad+\int_{\mathbb{R}} \frac{|t\alpha(\alpha+1)|\xi|^{\alpha-1}||\widehat{P_k\phi}(\xi)\psi_l(\xi-\xi_0)|}{|x-t(\alpha+1)|\xi|^{\alpha-1}\xi|^2}\, d \xi\\
	&\lesssim t^{-1}2^{-l/2+k(2-\alpha)}\|\partial\widehat{P_k\phi}\|_{L^2}
	+t^{-1}2^{-l+k(2-\alpha)}\|\widehat{P_k\phi}\|_{L^\infty},
	\end{align*}
	which gives 
	\begin{align*}
	\sum_{l\geq l_0+1}J_l\lesssim t^{-1/2}2^{k(1-\alpha/2)}\|\widehat{P_k\phi}\|_{L^\infty}
	+t^{-3/4}2^{-k(3\alpha/4-1/2)}2^k\|\partial\widehat{P_k\phi}\|_{L^2}.
	\end{align*}

\end{proof}

\section*{Acknowledgments}  The work of both authors was partially  supported by the ANR project ANuI (ANR-17-CE40-0035-02).


\begin{thebibliography}{000}
	
	\bibitem{BGLV}
	{\sc J. Bellazini, V. Georgiev, E. lenzmann and N. Visciglia}, {\it On traveling solitary waves and absence of small data scattering for nonlnear half-wave equation}, Comm. Math. Phys., {\bf 372} (2019), 713-732.
	
	
	\bibitem{BHL} 
	{\sc T. Boulenger, D. Himmelsbach and E. Lenzmann}, {\it Blow-up for fractional NLS}, J. Funct. Anal., {\bf 271} (2016), 2569-2603.
	
	\bibitem{MR3961297} 
	{\sc D. C\'{o}rdoba, and J. G\'{o}mez-Serrano, and  A. Ionescu}, \textsc{\it Global solutions for the generalized {SQG} patch equation}, Arch. Ration. Mech. Anal., {\bf 233} (2019), 1211-1251.
	
	
	\bibitem{CMMT}
	{\sc D. Cai, A. Majda, D. McLaughlin and E. Tabak}, {\it A one-dimensional model for dispersive wave turbulence}, Phys. D., {\bf 152-153} (2001), 551-572.
	
	
	\bibitem{Ca}
	{\sc T. Cazenave}, {\it Semilinear Schr\"{o}dinger equations}, Courant Lecture Notes in Mathematics, {\bf 10}, New York University, Courant Institute of Mathematical Sciences, New York; American Mathematical Society, Providence, RI, 2003.
	
	
	\bibitem{CHKL} 
	{\sc Y. Cho, G. Hwang, S. Kwon and S. Lee}, {\it Well-posedness and ill-posedness for the cubic fractional Schr\"{o}dinger equation}, Discrete Contin. Dyn. Syst., {\bf 35} (2015), 2863-2880.
	
	
	
	\bibitem{CP}
    {\sc A. Choffrut and O. Pocovnicu}, 
    {\it Ill-posedness of the cubic nonlinear half-wave equation and other fractional NLS on the real line}, Int. Math. Res. Not., {\bf 2018} (2018), 699-738.
	
	
	\bibitem{D1}
	{\sc V. Dinh}, {\it Well-posedness of nonlinear fractional Schr\"{o}dinger and wave equations in Sobolev spaces}, Int. J. Appl. Math., {\bf 31} (2018), 483-525.
	
	\bibitem{D2}
	{\sc V. Dinh}, {\it On blow-up solutions to the focusing mass-critical nonlinear fractional Schr\"{o}dinger equation}, Commun. Pure Appl. Anal., {\bf 18} (2019), 689-708.
	
	\bibitem{D3}
	{\sc V. Dinh}, {\it Blow-up criteria for fractional nonlinear Schr\"{o}dinger equations}, Nonlinear Anal. Real World Appl., {\bf 48} (2019), 117-140.
	
	\bibitem{Dod}\textsc{B. Dodson}, {\it Global well-posedness and scattering for the defocusing, mass-critical generalized KdV equation}, Ann. PDE, {\bf 3}, (2017), Paper No 5.
	
	\bibitem{FLPV}
	{\sc L. Farah, F. Linares, A. Pastor and N. Visciglia}, {\it Large data scattering for the defocusing supercritical generalized KdV equation}, Commun. Partial. Differ. Equ., {\bf 43} (2018), 118-157.
	
	\bibitem{FL}
	{\sc R. Frank and E. Lenzmann}, {\it On the uniqueness and non-degeneracy of ground states of 
	$(-\Delta)^s Q+Q-Q^{\alpha +1}=0\; \text{in}\; \R$}, Acta Math., {\bf 210}  (2013), 261--318.


    
    
    \bibitem{GTV}
    {\sc V. Georgiev, N. Tzvetkov and N. Visciglia}, {\it On th eregularity of the flow-map associated with the $1D$ cubic Half-Wave equation}, Differ. Integral Equ., {\bf 29} (2016), 183-200.
    
    
     \bibitem{GLPR}
     {\sc P. G\'erard, E. Lenzmann, O. Pocovnicu and P. Rapha\"{e}l}, {\it A two-soliton with transient turbulent regime for the cubic half-wave equation on the real line}, Ann. PDE., {\bf 4} (2018), Paper No.7.
     
     \bibitem{GG}
     {\sc P. G\' erard and S. Grellier}, {\it Effective integrable dynamics for a cubic nonlinear wave equation}, Anal. PDE, {\bf 5}, (2012), 1139-1155.
     
    
    \bibitem{GMS}
    {\sc P. Germain, N. Masmoudi, and J. Shatah}, {\em Global solutions for 3{D}
    quadratic {S}chr\"{o}dinger equations}, Int. Math. Res. Not., {\bf 2009} (2009), 414-432.
    
    \bibitem{MR3519470}
    {\sc P. Germain, F. Pusateri, and F. Rousset}, {\it Asymptotic stability of solitons for m{K}d{V}}, Adv. Math., {\bf 299} (2016), 272-330.
    
    
    \bibitem{GNT}
    {\sc S.~Gustafson, K.~Nakanishi, and T.-P.~Tsai}, {\em Global Dispersive Solutions for the Gross–Pitaevskii Equation in Two and Three Dimensions}, Annales Henri Poincar\'{e}, {\bf 8} (2007), 1303-1331.

   
	\bibitem{Guo}
	{\sc Z. Guo}, {\it Local well-posedness for dispersion generalized Benjamin-Ono equations in Sobolev spaces}, J. Differ. Equ.,  {\bf 252} (2012), 2053-2084.
	
	
	\bibitem{GZ} 
	{\sc Q. Guo and S. Zhu}, {\it Sharp criteria of scattering for the fractional NLS}, arXiv:1706.02549, (2017).	
		
	
	\bibitem{MR3462131}
	{\sc B.~Harrop-Griffiths}, {\it Long time behavior of solutions to the m{K}d{V}}, Commun. Partial. Differ. Equ., {\bf 41} (2016), 282-317.
	
	
	\bibitem{HN0} 
	{\sc N. Hayashi and P. Naumkin}, {\it Large time asymptotics of solutions to the generalized Benjamin-Ono equation},  Trans. Am. Math. Soc., {\bf 351} (1999), 109-130.
	
	
	\bibitem{HN1} 
	{\sc N. Hayashi and P. Naumkin}, {\it Large time behavior of solutions for the modified Korteweg-de Vries equation},  Int. Math. Res. Not., {\bf 1999} (1999), 395-418.
	
	
	
	\bibitem{HN2}
	{\sc N. Hayashi and P. Naumkin}, {\it Large time asymptotics for the fractional order cubic nonlinear Schr\"odinger equations}, Ann. Henri Poincar\'e,  {\bf 18} (2017), 1025-1054 
	

	\bibitem{HN5}
	{\sc N. Hayashi and P. Naumkin}, {\it Large time asymptotics for the fractional nonlinear Schr\"odinger equation}, Adv. Differ. Equ., {\bf 25} (2020), 31-80.
	
	
	\bibitem{HN3}
	{\sc J. Mendez-Navarro, P. Naumkin and Isahi S\'anchez-Su\'arez}, {\it Fractional nonlinear Schro\"odinger equation}, Z. Angew. Math. Phys., (2019) 70:168. 
	
	
	\bibitem{Him} 
	{\sc D. Himmelsbach}, {\it Blowup, solitary waves and scattering for the fractional nonlinear Schr\"{o}dinger equation}, Inauguraldissertation, gBasel 2017.
	
	
	\bibitem{HS} 
	{\sc T. Hong and Y. Sire}, {\it On fractional Schr\"{o}dinger equations in Sobolev spaces}, Commun. Pure Appl. Anal., {\bf 14} (2015), 2265-2282.
	
	\bibitem{HS2} 
	{\sc T. Hong and Y. Sire}, {\it A new class of traveling solitons for cubic fractional nonlinear Schr\"{o}dinger equation}, Nonlinearity, {\bf 30} (2017), 1262-1286
	
	\bibitem{MR3121725}
	{\sc A. Ionescu and F. Pusateri}, {\it Nonlinear fractional Schr\"{o}dinger equations in one dimension}, J. Funct. Anal., {\bf 266} (2019), 139-176.
	
	
	\bibitem{IT1}
	\textsc{M. Ifrim, D. Tataru}, {\it Global bounds for the cubic nonlinear Schr\"{o}dinger equation (NLS) in one space dimension}, Nonlinearity, {\bf 28} (2015), 2661–2675.
	
	
	
	\bibitem{KLPS} 
	{\sc C. Klein, F. Linares,  D.Pilod and J.-C. Saut},  {\it On Whitham and related equations},  Studies in Appl. Math.,  {\bf 140} (2018),  133-177
		
	
	\bibitem{KLS}
	{\sc K. Kirkpatrick, R. Lenzmann and G. Staffilani}	, {\it On the continuum limit for discrete NLS with long-range interactions}, Comm. Math. Phys., {\bf 317} (2013), 563-591.
	
	
	
	\bibitem{KSW}
	{\sc C. Klein, J.-C. Saut and Y. Wang}, {\it On the modified fractional Korteweg-de Vries and related equations}, arXiv:2010.05081, (2020).
	
	
	\bibitem{KSM}
	{\sc C. Klein, C. Sparber and P. Markowich}, {\it Numerical study of fractional Nonlinear Schr\"odinger equations}, Proc. R. Soc. A  {\bf 470}: 20140364, http://doi.org/10.1098/rspa.2014.0364.
	
	
	\bibitem{KT}
	{\sc C. Kenig and T. Takaoka}, {\it Global well-posedness of the modified Benjamin-Ono equation with initial data in $H^{1/2}$}, Int. Math. Res. Not., {\bf 2006} (2006), 1-44.
	
	\bibitem{KLR}
	{\sc J. Krieger, E. Lenzmann and P. Rapha\"{e}l}, {\it Non dispersive solutions to the $L^2$ critical half-wave equation}, Arch Ration Mech Anal., {\bf 209} (2012), 61-129.
	
	\bibitem{Lan}
	{\sc Y. Lan} {\it Blow-up dynamics for $L^2-$ critical fractional Schr\"{o}dinger equations}, arXiv:1908.09561, (2019).
	
	
	\bibitem{La1}
	{\sc D. Lannes}, {\it Water waves: mathematical theory and asymptotics},  Mathematical Surveys and Monographs, {\bf 188} (2013), AMS, Providence.
	
	\bibitem{Laskin}
	{\sc N. Laskin}, {\it Fractional Schr\"odinger equation}, Phys. Rev. E, 
	{\bf 66}, 66.056108, (2002).
	
	\bibitem{LPS3}
	{\sc F. Linares, D. Pilod and J.-C. Saut}, {\it Remarks on the orbital stability of ground state solutions of fKdV and related equations},  
	Adv. Differ. Equ., {\bf 20} (2015), 835-858.
	

	
    \bibitem{LPS}
	{\sc F. Linares, D. Pilod and J.-C. Saut}, {\it Dispersive perturbations of Burgers and hyperbolic equations I: local theory}, SIAM J. Math. Anal., {\bf 46} (2014), 1505-1537.
	
	
	\bibitem{MMT}
	{\sc A. Majda, D. McLaughlin and E. Tabak}, {\it A one-dimensional model for dispersive wave turbulence}, J. Nonlinear Sci., {\bf 6} (1997), 9-44.
	
	\bibitem{MP}
	{\sc Y. Martel and D. Pilod}, {\it Construction of a minimal mass blow-up solution to the modified Benjamin-Ono equation}, Math. Annal., {\bf 369} (2017), 153-245.	
	
	
	\bibitem{MPV}
	{\sc L. Molinet, D. Pilod and S. Vento}, {\it On well-posedness for some dispersive perturbations of the Burgers equation}, Ann. Inst. H. Poincar\' e Anal. Non Lin\' eairre, {\bf 35} (2018), 1719-1756.
	
	
	\bibitem{NRT} 
	{\sc A. Nachman, I. Regev and D. Tataru}, {\it  A nonlinear Plancherel theorem with applications to global well-posedness for the defocusing Davey-Stewartson equation and to the inverse boundary value problem of Calder\' on}, Invent. Math., {\bf 220} (2020), 395-451.
	
	\bibitem{Nau}
	{\sc  P. Naumkin}, {\it Fractional nonlinear Schr\"odinger equation of order \(\alpha\in(0,1)\)}, J. Differ. Equ., {\bf 269} (2020), 5701-5729.
	

	
	\bibitem{OS}
	{\sc C. Obrecht and J.-C. Saut}, {\it Remarks on the full-dispersion Davey-Stewartson systems}, Commun. Pure Appl. Anal., {\bf 14} (2015), 1547-1561.
	
	\bibitem{Ozawa}
	{\sc T.~Ozawa}, {\it Long range scattering for nonlinear {S}chr\"{o}dinger equations in one space dimension}, Comm. Math. Phys., {\bf 139} (1991), 479-493.
	
	\bibitem{Pe} 
	{\sc P. Perry},  {\it Global well-posedness and long time asymptotics for the defocussing Davey-Stewartson II equation in $H^{1,1}(\R^2)$}, J. Spectr. Theory, {\bf 6} (2014), 429-481.

    \bibitem{Poco} 
    {\sc O. Pocovnicu}, {\it First and second order approximations for a nonlinear wave equation}, J. Dyn. Diff. Eq. {\bf 25} (2013), 305-333.
	
	\bibitem{SW1}
	{\sc J.-C. Saut and Y. Wang}, {\it Long time behavior of the fractional Korteweg-de Vries equation with cubic nonlinearity}, Discrete Contin. Dyn. Syst., https://doi.org/10.3934/dcds.2020312.
	
	\bibitem{SW}
	{\sc J.-C. Saut and Y. Wang}, {\it The wave breaking for Whitham-type equations revisited}, arXiv:2006.03803, (2020).
	
	\bibitem{SunTz}
	\textsc{C. Sun and N. Tzvetkov}, {\it Gibbs measure dynamics for the fractional NLS}, SIAM J. Math. Anal., {\bf 52} (2020), 4638-4704.
	
   \bibitem{SunTz2}
    {\sc C. Sun and N. Tzvetkov}, {\it Refined probabilistic global well-posedness for weakly dispersive NLS}, arXiv:2010.13065, 2020.

   \bibitem{Su4}
   {\sc L. Sung}, {\it Long-Time Decay of the Solutions of the 
   	Davey-Stewartson II Equations}. J. Nonlinear Sci., {\bf 5} (1995), 433-452.

   \bibitem{Thi}
   {\sc J. Thirouin}, {\it On the growth of solutions of the fractinal defocusing nonlinear Schr\"{o}dinger equation on the circle}, Ann. Inst. H. Poincar\' e {\bf 34} (2017).	509-531.
   
   \bibitem{ZGPD}
   {\sc V. Zakharov, P. Guyenne, A. Pushkarev and F. Dias},{\it Wave turbulence in one-dimensional models}, Phys. D., {\bf 152-153} (2001), 573-619.
	
\end{thebibliography}
\end{document}